\documentclass[11pt, reqno]{amsart}
\usepackage{fullpage}

 \newif\ifHideFoot
\HideFoottrue

\ifHideFoot
\else
\usepackage[notref,notcite]{showkeys}
\fi 

\usepackage{amssymb,amsmath,amsthm,amscd,mathrsfs,graphicx, color}
\usepackage[cmtip,all,matrix,arrow,tips,curve]{xy}
 \usepackage{hyperref}
 \usepackage{comment}
\usepackage[usenames,dvipsnames]{xcolor}
\usepackage{xypic}
\hypersetup{colorlinks=true,citecolor=ForestGreen,linkcolor=Maroon,urlcolor=NavyBlue}
 \usepackage{mathtools}
\usepackage{mathpazo}
\usepackage[normalem]{ulem}
\usepackage{thmtools}
\usepackage{cleveref}
\usepackage{enumitem}
\usepackage{tikz-cd}
\usepackage{pgfplots}
  
\numberwithin{equation}{section}
\newtheorem{teo}{Theorem}[section]
\newtheorem{pro}[teo]{Proposition}
\newtheorem{lem}[teo]{Lemma}
\newtheorem{cor}[teo]{Corollary}

\newtheorem{teoalpha}{Theorem}

\newtheorem{coralpha}[teoalpha]{Corollary}

\theoremstyle{definition}

\theoremstyle{remark}
\newtheorem{rem}[teo]{Remark}

\ifHideFoot

\newcommand{\Yano}[1]{}
\newcommand{\Shend}[1]{}

\else 

\newcommand{\marg}[1]{\normalsize{{
\color{red}\footnote{{\color{blue}#1}}}{\marginpar[\vskip
-.25cm{\color{red}\hfill\tiny\thefootnote$\implies$}]{\vskip
-.2cm{\color{red}$\impliedby$\tiny\thefootnote}}}}}
\newcommand{\Yano}[1]{\marg{(Yano) #1}}
\newcommand{\Shend}[1]{\marg{(Shend) #1}}

\fi

\newcommand{\m}[1]{\mathcal{#1}}

\newcommand{\X}{\mathcal{X}}
\newcommand{\XX}{\mathcal{X}}

\newcommand{\Y}{\mathcal{Y}}

\newcommand{\Q}{\mathbb{Q}}

\newcommand{\ra}{\rightarrow}

\title[Moduli of varieties of general type]{Moduli spaces of varieties of general type are naturally of log general type}

\author{Sebastian Casalaina-Martin}
\address{University of Colorado, Department of Mathematics, 
Boulder, CO 80309, USA }
\email{casa@math.colorado.edu}

\author{Shend Zhjeqi}
\address{University of Michigan, Department of Mathematics, 
Ann Arbor, MI 48109, USA }
\email{shendzh@umich.edu}

\thanks{Research of the first named author is supported in part by a grant from the Simons Foundation (SFI-MPS-TSM-00013682). The second named author was partially supported by the Simons Collaboration grant Moduli of Varieties.}

\date{\today}

\begin{document}

\begin{abstract}
We generalize work of Popa--Schnell on Viehweg hyperbolicity for smooth projective varieties to the case of smooth Deligne--Mumford stacks.  The results of Popa--Schnell build on results of Viehweg--Zuo, utilize a result of Campana--P\u{a}un, and extend results of Kebekus--Kov\'acs and Patakfalvi.  We show that if a smooth proper DM stack with projective coarse moduli space parameterizes a family of varieties of general type having maximal variation, then the natural log canonical bundle of the stack is big.  We also consider implications for the coarse moduli space of the stack, and apply the results to several standard moduli spaces.
\end{abstract}

\maketitle

\ifHideFoot
\setcounter{page}{1}
\fi

\section*{Introduction}

Recent work of a number of authors \cite{VZ03, KK08, Pat12, CPFol19, PS17, WW23} has established Viehweg hyperbolicity for families of smooth projective varieties of general type.  
These results build on strategies developed in \cite{Zuo00negativity, VZ01isotriv, VZ02, VZ03},   as well as positivity  results for variations of Hodge structures  due to Griffiths and  \cite{fujita_fiber_spaces, Kawamata, FF14variationsHMS, FFS14Remarks, Zuo00negativity,PW16, PTW18}, and  for push forwards of relative canonical bundles due to \cite{kollar_sub, EV90effective, KP17proj}.
Specifically, in the paper of 
Popa--Schnell \cite{PS17}, they establish that if $f:Y\to X$ is an algebraic fiber space of maximal variation between smooth projective varieties, $\Delta\subseteq X$ is any divisor containing the discriminant, and the geometric generic fiber is of general type (or simply admits a good minimal model), then $K_X+\Delta$ is big. 
 
While one can apply this result to coarse moduli spaces of smooth moduli stacks by ignoring loci with extra automorphisms,  our goal in this paper is to generalize the result to smooth proper DM stacks with projective coarse moduli space, so that we can apply the results directly to moduli stacks.  This seems more natural in the context of moduli spaces, and gives sharper statements when applied to the coarse moduli space of the stack.  We discuss this comparison further below.
This paper is the fifth in a series with the aim of generalizing results of \cite{PS17, WW23}, on Viehweg hyperbolicity, to
the case of Deligne–Mumford stacks; the previous articles in this series are 
\cite{CMZpositivity, CMZslope_stability, CMZfoliations, CMZlog_general_type_KSBA}. 

Our main result is 
\Cref{T:main-goodMM}, below, which generalizes \cite[Thm.~A]{PS17} to the case of DM stacks:

\begin{teoalpha}[Viehweg hyperbolicity for DM stacks] 
\label{T:main-goodMM} 
Let $f:\mathcal Y\to \mathcal X$ be a surjective schematic projective morphism of smooth proper integral DM stacks over $\mathbb C$ with projective coarse moduli spaces, and assume  $f$ has connected fibers.  Let $\mathbf \Delta\subseteq \mathcal X$ be any reduced divisor containing the discriminant locus   of  $f$.
Assume the following:
\begin{enumerate}

\item  The geometric generic fiber of $f$ is of general type, or, more generally,  the geometric generic fiber of $f$ admits a good minimal model (\cite[p.4]{kawamata_kod_dim}), 

\item $f$ has maximal variation, in the sense that $\operatorname{Var}(f)=\dim \mathcal X$.   
\end{enumerate}
 If $\mathcal X$ has generically trivial stabilizers, then  $K_{\mathcal X}+\mathbf \Delta$ is big.
\end{teoalpha}

We refer the reader to \cite[Def.~2.4]{CMZpositivity} for the definition of a big line bundle on $\mathcal X$; in short, this means that a positive tensor power of the line bundle descends to a big line bundle on the coarse moduli space.  The variation of $f$ is the variation of the family after pull back by any \'etale cover of $\mathcal X$ (see \S \ref{S:variation}).  
The case  of \Cref{T:main-goodMM} 
where $\mathcal X=[V/G]$ is the quotient of a smooth projective variety $V$ by a finite group $G$ (even without the assumption on generically trivial stabilizers) follows directly from \cite[Thm.~A]{PS17}; the content of our result is to extend the result to more general smooth Deligne--Mumford stacks.
The assumption in \Cref{T:main-goodMM} that $\mathcal X$ have generically trivial stabilizers is used only in a construction in the proof of \Cref{P:PTW18-PA.1}, where certain global generation properties are needed.  We  expect this hypothesis can be removed using a different approach.     The proof of \Cref{T:main-goodMM} is given in \S \ref{S:ProofMain}.

\medskip

Our proof of \Cref{T:main-goodMM}   closely follows the strategy developed in \cite{PS17},
which in turn builds on the strategies developed by  Viehweg and Zuo in  \cite{Zuo00negativity, VZ01isotriv, VZ02, VZ03}. 
   The first observation is that one can replace $\mathcal X$  by a log resolution (see, e.g.,  \cite[Lem.~1.9]{CMZlog_general_type_KSBA}), so we may as well assume that the divisor $\mathbf \Delta$ is a normal crossings divisor. 
The main goal is then to establish the existence of a so-called Viehweg--Zuo sheaf, i.e., a sheaf $\mathcal H$ with big determinant that sits in a short exact sequence 
$$
\xymatrix{
0 \ar[r]& \mathcal H \ar[r]& (\Omega^1_{\mathcal X}(\log \mathbf \Delta))^{\otimes s} \ar[r] & \mathcal Q \ar[r]& 0,
}
$$
where  $s$ is a positive integer, and for simplicity, one may always replace $\mathcal H$ with its saturation.  It follows in the  case of varieties from 
a theorem of Campana--P\u{a}un  \cite[Thm.~7.6, Thm.~1.2]{CPFol19},   and from   
\cite[Thm.~A]{CMZfoliations} in the case of stacks, that   $\det \mathcal Q$ is pseudo-effective.  
Taking determinants in the short exact sequence above, one obtains
that a positive multiple of $K_{\mathcal X}+\mathbf \Delta$ is ``big plus pseudo-effective'', and therefore is big  (e.g., \cite[Cor.~B]{CMZfoliations}).

\medskip 

   The proof of \Cref{T:main-goodMM}  is therefore reduced to constructing Viehweg--Zuo sheaves.  Our approach is to follow the strategy of \cite{PS17}, which builds on techniques developed in \cite{VZ01isotriv, VZ02, VZ03}.  
   The basic idea is to try to show there exists a big sub-line bundle contained in the  lowest piece of the Hodge filtration for the middle variation of Hodge structure associated to $f$, or more precisely,   the lowest piece of the Hodge filtration for the  Hodge module $\mathsf M$ obtained as the torsion-free quotient of $\mathcal H^0f_+\mathbb Q_{\mathcal Y}[\dim \mathcal Y]$.  
To make this work, one in fact needs to make the variation of Hodge structure more positive, by considering fibered products and cyclic covers.  Much of the necessary work for extending these techniques to stacks was accomplished in \cite{CMZlog_general_type_KSBA}, however, there are some key points that are different in the case considered in this paper.   The main new additions here are extending work of \cite{VZ03} on mild reduction to the case of DM stacks, as well as work of \cite{Vkod1} regarding some applications of relative duality (see \S \ref{S:mild_redux}).
The culmination of all of this is provided by the results \Cref{P:PTW18-PA.1} and \Cref{P:PTW18-PA.4}, which allow for the necessary fibered product and cyclic covering constructions.  

 Once one has \Cref{P:PTW18-PA.1} and \Cref{P:PTW18-PA.4},  then results of \cite{PS17}, generalized to the case of stacks in \cite[Thm.~4.3 and Thm.~5.1]{CMZlog_general_type_KSBA},  allow one to 
  replace the original family with the family obtained by the fibered product and cyclic covering construction, and then  
use the theory of Hodge modules to  construct an effective snc divisor $\mathbf E\subseteq \mathcal X$, containing $\mathbf \Delta$, together with a graded logarithmic Higgs bundle on $\mathcal X$
$$
\theta_\bullet: \mathcal E_\bullet \longrightarrow \mathcal E_\bullet \otimes \Omega^1_{\mathcal X}(\log \mathbf E)
$$ 
containing a graded sub-sheaf $\mathcal F_\bullet\subseteq \mathcal E_\bullet$ such that $\theta_p(\mathcal F_p)\subseteq \mathcal F_{p+1}\otimes\Omega^1_{\mathcal X}(\log \mathbf \Delta)$, and such  that the first nonzero sheaf in the filtration $\mathcal F_\bullet$ is  a big line bundle.
This step is one  of several  key additions that Popa--Schnell \cite{PS17} made to the strategy for constructing Viehweg--Zuo sheaves, by using  Saito's theory of Hodge modules, and the compatibility of Hodge and
$V$-filtrations, to show  the existence of the  Higgs sheaf $\mathcal F_\bullet$  that is 
logarithmic with respect to the discriminant locus $\mathbf \Delta$, as opposed to the
singular locus of the torsion-free quotient of the Hodge module $\mathcal H^0f_+\mathbb Q_{\mathcal Y}[\dim \mathcal Y]$, where the latter could be bigger.

At this point, the proof relies on some of our previous results in \cite{CMZpositivity, CMZfoliations}, generalizing some results of \cite{PW16,PS17} to the case of DM stacks.  
 More precisely, the existence of the  
 graded logarithmic Higgs bundle $\mathcal E_\bullet$ 
  and the graded subsheaf $\mathcal F_\bullet$  
   implies from \cite[Thm.~B]{CMZpositivity} 
   the existence of a Viehweg--Zuo sheaf, completing the construction.

\medskip 
The above results also have implications for the coarse moduli spaces of the stacks. 
We let  $\pi:\mathcal X\to X$ denote  the canonical morphism to the coarse moduli space, with ramification divisor $\mathbf R$ on $\mathcal X$, and branch $\mathbb Q$-divisor $R$ on $X$ ($\mathbf R=\pi^*R$ and $K_{\mathcal X} = \pi^*K_X+\mathbf R=\pi^*(K_X+R)$).
Given a divisor $\mathbf \Delta$ on $\mathcal X$, we let $\Delta$ be the $\mathbb Q$-divisor on $X$ such that $\mathbf \Delta=\pi^*\Delta$.  
All together, we have
$
K_{\mathcal X}+\mathbf \Delta = \pi^*(K_X+R+\Delta)
$. 
Consequently,  $K_{\mathcal X}+\mathbf \Delta $ is big if and only if $K_X+R+\Delta$ is big.  

In the situation we are in, i.e., working with a  normal integral complex projective  $\mathbb Q$-factorial variety $X$ with an effective $\mathbb Q$-divisor $\Delta$ on $X$ with coefficients in $(0,1]$, we will say that the pair $(X,\Delta)$ is of log general type if for any log resolution $\mu:X'\to X$ of the pair $(X,\Delta)$, taking $\Delta'$ to be the union of the strict  transform of $\Delta$ with the \emph{reduced exceptional locus of $\mu$}, one has that $K_{X'}+\Delta'$ is big.  If $(X,\Delta)$ is log canonical and $K_X+\Delta$ is big, then it follows that $(X,\Delta)$ is of log general type (see also \cite[Prop.~1.27]{kollar_singularities_MMP}).
With this terminology, we have the following:

\begin{coralpha}\label{main-coarseMS-cor}
In the situation of \Cref{T:main-goodMM}, one has that $(X,R+\Delta)$ is of log general type.   
\end{coralpha}

The proof of \Cref{main-coarseMS-cor} is given in \S \ref{S:ProofMain}.
With this result, one can clearly explain the difference between our result  and  \cite[Thm.~A]{PS17}.  Assume we are in the situation where $\mathcal X$ and $\mathbf \Delta$ are as in  \Cref{T:main-goodMM}.
By virtue of the assumption that the stack $\mathcal X$ has trivial automorphisms at the generic point, the stack  $\mathcal X$ and the coarse moduli space $X$ agree on $X-(R\cup \Delta)$ outside of codimension $2$, and then, using the family over this open variety,  \cite[Thm.~A]{PS17} implies that $K_X+\lceil R+\Delta \rceil$ is big. From this it follows that $(X,\lceil R+\Delta \rceil)$ is of log general type (see the proof of \Cref{main-coarseMS-cor}).   
There seem to be two benefits to the stronger statement in \Cref{main-coarseMS-cor}, i.e.,  regarding the pair $(X,R+\Delta)$ rather than $(X,\lceil R+\Delta \rceil)$.  First, the coefficients of 
$R+\Delta$ may be smaller than those of $\lceil R+\Delta \rceil$, so that one has a more precise statement about how much of the boundary divisor should be  added to the canonical divisor to  obtain a big  divisor.  Second, 
 it seems plausible that there could be situations where $(X,R+\Delta)$  was log canonical but $(X,\lceil R+\Delta \rceil)$ was not; note that if $\mathbf \Delta\subseteq \mathcal X$ is snc, then    $(X,R+\Delta)$ is log canonical.

A familiar elementary example makes the difference between the results apparent.  Consider the family of elliptic curves $\{zy^2=x(x-z)(x-\lambda z)\}\subseteq \mathbb P^1_\lambda \times \mathbb P^2$ parameterized by the $\lambda$-line $\mathbb P^1_\lambda$, as well as the stack $[\mathbb P^1_\lambda /S_3]$, where $S_3$  acts in the natural way, with orbits 
$\{\lambda, 1-\lambda, \frac{1}{\lambda}, \frac{1}{1-\lambda}, \frac{\lambda-1}{\lambda}, \frac{\lambda}{\lambda-1}\}$.
The coarse moduli space $[\mathbb P^1_\lambda /S_3]\to \mathbb P^1_j$ can be taken to be the $j$-line, given by taking the $j$-invariant of the corresponding elliptic curve.  
We view $[\mathbb P^1_\lambda /S_3]$  as the base of a family of  $K$-trivial varieties, in which case both \Cref{main-coarseMS-cor} and \cite[Thm.~A]{PS17} apply to $\mathbb  P^1_j$, since the universal family descends over $\mathbb P^1_j-\{p_0,p_{1728}, p_\infty\}$, where here $p_j$ denotes the point parameterizing elliptic curves with $j$-invariant $j$.  
 For this example, we have 
 $R=(1-\frac{1}{3})p_{0}+ (1-\frac{1}{2})p_{1728}+(1-\frac{1}{2})p_{\infty}$ and $\Delta = \frac{1}{2}p_\infty$, 
 so that taking sums we have
$
 R+\Delta  \sim_{\mathbb Q} \frac{13}{6}p_0 \ \ \ \ \text {and} \ \ \ \ \lceil R+\Delta \rceil  \sim_{\mathbb Q} 3p_0
 $.
From \Cref{main-coarseMS-cor}, the assertion is that
$
K_{\mathbb P^1}+(R+\Delta) = K_{\mathbb P^1}+\frac{13}{6}p_0\sim_{\mathbb Q}\frac{1}{6}p_0
$
is big, while the results \cite[Thm.~A]{PS17}  imply that
$
K_{\mathbb P^1}+\lceil R+\Delta\rceil = K_{\mathbb P^1}+3p_0\sim p_0
$
is big. In this case, the two statements, that $\frac{1}{6}p_0$ and $p_0$ are big, are equivalent, but the former shows that it suffices to add only $\frac{13}{6}p_0$ to the canonical bundle to obtain a big line bundle (as opposed to $3p_0$ in the latter).  We discuss some further applications of \Cref{T:main-goodMM} and \Cref{main-coarseMS-cor} to some other standard  examples in \S \ref{S:examples}.

\subsection*{Acknowledgements}
The first named author thanks Mihnea Popa for  conversations on the topic, which led to this project.  He also thanks  Jonathan Wise and David Rydh for helpful conversations about the geometry of stacks, as well as Klaus Hulek  for conversations about the geometry of the moduli space of abelian varieties.  The second named author thanks his advisor, Mircea Musta\c{t}\u{a}, for useful discussions and all the support provided.  The authors are also grateful to the organizers of the Simons Collaborations on Moduli of Varieties Workshop at the University of Utah in November 2024, where their work on this project began.

\section{Preliminaries}\label{S:prelims}

\subsection{Terminology}
We use the same conventions as in \cite{CMZpositivity, CMZslope_stability, CMZfoliations}.
We work over $\mathbb C$.  
A \emph{variety} is an integral separated scheme of finite type over $\mathbb C$. 
An \emph{alteration} $X'\to X$ is a surjective projective generically
finite   morphism of schemes over $\mathbb C$.  
We use the definition of a \emph{Deligne--Mumford (DM) stack} in \cite[Def.~4.1]{LMB}. Note that this differs from the definition in \cite{stacks-project} in that there is the additional hypothesis in \cite[Def.~4.1]{LMB} that the diagonal be representable, separated, and quasi-compact.  We direct the reader to \cite[App.~B]{CMW18} for a discussion of the relationship among various definitions of DM stacks in the literature (see in particular \cite[Fig.~1]{CMW18}).  We emphasize that, with the definition of DM stack that we are using, a morphism from a scheme to a DM stack is schematic (representable by schemes); see e.g., \cite[Lem.~B.20 and Lem.~B.12]{CMW18}.

\subsection{Structure of DM stacks}\label{S:DM-intro}
Again, we use the same conventions as in \cite{CMZpositivity, CMZslope_stability, CMZfoliations}.
  The general set-up will be a smooth proper (resp.~separated) integral DM stack $\mathcal X$ of finite type over $\mathbb C$ with coarse moduli space $\pi: \mathcal X\to X$, with the added assumption that the algebraic space $X$ be a projective (resp.~quasi-projective) variety. 
Recall that such a stack admits a finite flat 
morphism $q:V\to \mathcal X$ from a smooth projective  (resp.~quasi-projective) 
variety $V$  (\cite[Thm.~1]{KV04} and \cite[Thm.~4.4]{kresch09}, see also \cite[\href{https://stacks.math.columbia.edu/tag/03B6}{\S 03B6}]{stacks-project}); note that $q$ is schematic and projective.  
Note also that by the proof in \cite[Thm.~1]{KV04}, one can take $q$ to be \'etale over any given finite collection of points of $\mathcal X$. 
In this situation we have that $X$ is normal, $\mathbb Q$-factorial, with at worst klt singularities.  The morphism $\pi:\mathcal X\to X$ is flat over the smooth locus of $X$; flatness is an \'etale local property, and so it suffices to consider the case $\mathcal X=[U/G]$ for some smooth variety $U$ and a finite group $G$.  Then, from say \cite[Cor.~14.12]{GW20}, it suffices to show that $U\to U/G$ is flat over the smooth locus of the quotient, which follows from the miracle of flatness \cite[Thm.~23.1, p.179]{matsumura}.

For brevity,  we will say that such a stack $\mathcal X$ is a global finite quotient stack if there is a smooth projective (resp.~quasi-projective) variety $V$ over $\mathbb C$ and a finite algebraic group $G$ over $\mathbb C$ acting on $V$ such that $\mathcal X\cong [V/G]$; note that this implies that $X=V/G$.  For context, recall that  $\mathcal X$ is a global finite quotient stack  if and only if there exists a smooth projective (resp.~quasi-projective) variety  $V'$ and a finite \emph{\'etale} morphism $q':V'\to \mathcal X$ (see the \emph{proof} of \cite[Thm.~(6.1)]{LMB}).

\subsection{Variation}\label{S:variation}

Let $f:\mathcal Y\to \mathcal X$ be a schematic projective morphism of smooth proper integral DM stacks over $\mathbb C$ with projective coarse moduli spaces, and assume  $f$ has geometrically connected generic fiber.  
Let $q':V'\to \mathcal X$ be a generically finite morphism from a smooth projective variety $V'$, let $Y'$ be a resolution of singularities of normalization of the main component of the  fibered product $V'\times_{\mathcal X}\mathcal Y$, and let $f':\mathcal Y'\to V'$ be the induced morphism of smooth projective varieties.  
We define the \emph{variation of $f$}, denoted $\operatorname{Var}(f)$,  to be the variation of $f'$.  This is independent of the choices, as one can always find common log resolutions.  We also have that the variation of $f$ agrees with  the variation of the family obtained after pull back by any \'etale cover of $\mathcal X$.

\subsection{Semi-stable reduction in codimension 1}
 We will want a version of semi-stable reduction in codimension one for stacks, generalizing \cite[Prop.~6.1]{Vkod1}:  
\begin{pro}

\label{P:PV6.1}
Let $f:\mathcal Y\to \mathcal X$ be a schematic projective morphism of smooth proper integral DM stacks over $\mathbb C$, such that $\mathcal X$ has a projective coarse moduli space.  Then there is a commutative diagram
$$
\xymatrix{
\widetilde {Y} \ar[r] \ar[d]_{\tilde f} & \mathcal Y \ar[d]^f\\
\widetilde { X} \ar[r]_\tau& \mathcal X
}
$$
with $\tau$ schematic projective and generically finite,
$\widetilde {X}$ and
$\widetilde {Y}$ smooth  projective varieties, and after removing
a closed substack  $\mathcal Z$ of codimension at least $2$ in $\mathcal X$, the morphism $\tau$  is finite and flat and $\tilde f$ is semi-stable. 
\end{pro}

\begin{proof}
Take a base change over  a schematic finite flat projective morphism $q:V\to \mathcal X$, where $V$ is a smooth projective variety.  Since $f$ is assumed to be schematic projective, we have that $Y_V:=V\times_{\mathcal X} \mathcal Y$ is projective.  
Applying the standard results for varieties (e.g., take a resolution of singularities $Y_V'\to Y_V$ and apply \cite[Lem.~6.1]{Vkod1} to $Y_V'\to V$), one obtains the diagram above with the asserted properties.
\end{proof}

\subsection{Weak semi-stable reduction}

The same strategy used in the proof of \Cref{P:PV6.1} provides a version of weak semi-stable reduction extending \cite[Thm.~0.3]{AK2000} to the setting of stacks:

\begin{pro}[{\cite[Thm.~0.3]{AK2000}}]
  \label{P:AK2000-Thm03}
Let $f:\mathcal Y\to \mathcal X$ be a schematic projective morphism of smooth proper integral DM stacks over $\mathbb C$, with geometrically connected fibers, such that $\mathcal X$ has a projective coarse moduli space.  Then there is a commutative diagram
\begin{equation}\label{E:AK2000-diag}
\xymatrix{
\widetilde {Y} \ar[r] \ar[d]_{\tilde f} & \mathcal Y \ar[d]^f\\
\widetilde { X} \ar[r]_\tau& \mathcal X
}
\end{equation}
with $\tau$ schematic projective and generically finite,
$\widetilde {X}$ and
$\widetilde {Y}$ projective varieties, $\widetilde X$ smooth, and  $\tilde f$ weakly semi-stable (see \cite[Def.~0.1]{AK2000}).
\end{pro}

\begin{proof}
  Indeed, after finite flat base change by a morphism $q:V\to \mathcal X$, where $V$ is a projective variety (\cite[Thm.~1]{KV04} and \cite[Thm.~4.4]{kresch09}, see also \cite[\href{https://stacks.math.columbia.edu/tag/03B6}{\S 03B6}]{stacks-project}), one can immediately employ \cite[Thm.~0.3]{AK2000}.
\end{proof}

\begin{rem}
Using  \cite[Thm.~2.7, Thm.~4.5, and Rem.~4.6]{ALT}, one can in fact take $\widetilde Y$ to be smooth and $\tilde f$ to be semi-stable.  
\end{rem}

\subsection{Kawamata's covering trick via root stacks}\label{S:KawamataCov}

Recall Kawamata's covering trick for varieties \cite[Cor.~2.6]{Vmoduli}:
\emph{Let $\tau: X'\to  X$ be a finite surjective morphism of quasi-projective  varieties.  
Assume that $X$ is smooth, and that for some normal crossing divisor
$ D = \sum_{j=1}^r D_j$ on $X$, the covering $\tau^{-1}(X- D)\to X-D$ 
 is \'etale.
Then there
exists a \emph{finite} morphism  $\gamma' : Z'\to X'$ from a  \emph{smooth} quasi-projective variety $Z'$, with $\gamma'$  \'etale over  $\tau^{-1}(X- D)$.}  

The basic strategy of the argument for the covering trick is to take roots of the divisors $D_j$ to obtain a morphism $\gamma:Z\to X$ so that the base change $Z'=Z\times_XX'\to Z$ is \'etale outside of a codimension-$2$ locus,  and then use purity of the branch locus to obtain that $Z'$ is smooth.  Since there may not exist roots of the $D_j$, the trick is to use sufficiently general hyperplane sections to add to the $D_j$, to ensure that one does have  roots, and that the branch locus is still normal crossings.  For stacks, we do not in general have such very ample divisors at our disposal, but, we can utilize the root stack construction instead:

\begin{lem}[Kawamata's covering trick]\label{C:KawCovTrick}
Let $\tau: \mathcal X'\to \mathcal X$ be a finite surjective morphism of  integral separated DM stacks of finite type over $\mathbb C$ with (quasi-)projective coarse moduli spaces.  
Assume that $\mathcal X$ is smooth, and that for some normal crossing divisor
$\mathcal D = \sum_{j=1}^r\mathcal D_j$ on $\mathcal X$, with the $\mathcal D_j$ irreducible, the covering $\tau^{-1}(\mathcal \mathcal X-\mathcal D)\to \mathcal X-D$ is \'etale.
Then there
exists a commutative diagram
\begin{equation}\label{E:C:KawCovTrick}
\xymatrix{
\mathcal Z'\ar[r]^{\gamma'}\ar[d]_{\tau'}& \mathcal X' \ar[d]^\tau\\
\mathcal Z\ar[r]^\gamma & \mathcal X
}
\end{equation}
where $\gamma$ is a composition of root stack constructions taking roots of the $\mathcal D_j$ (and their strict transforms), and $\tau'$ is \'etale.  In particular,  $\mathcal Z'$ is smooth with (quasi-)projective coarse moduli space, the divisor $\gamma^{-1}\mathcal D$ is normal crossing,  and $\gamma'$ is finite and  \'etale over  $\tau^{-1}(\mathcal X-\mathcal D)$.
\end{lem}

\begin{proof}
The strategy of the proof is the same as  \cite[Cor.~2.6]{Vmoduli}.  For $j=1,\dots,r$, choose
$$
N_j:=\operatorname{lcm}\{e(\mathbf \Delta^i_j) : \mathbf \Delta^i_j \text { is a component of } \tau^{-1}(\mathcal D_j)\},
$$
where $e(\mathbf \Delta^i_j)$ denotes the ramification index of $\mathbf \Delta^i_j$ over $\mathcal D_j$.  
Let $\gamma_1:\mathcal Z_1\to \mathcal X$ be the root stack obtained by taking the $N_1$-th root of $\mathcal D_1$.  A local computation, identical to the case of varieties, shows that by taking $\mathcal Z_1'$ to be an irreducible component of the normalization of $\mathcal Z_1\times_\mathcal X \mathcal X'$ that dominates $\mathcal Z_1$ and $\mathcal X'$, one obtains a diagram  
$$
\xymatrix{
\mathcal Z_1'\ar[r]^{\gamma'_1}\ar[d]_{\tau'}& \mathcal X' \ar[d]^\tau\\
\mathcal Z_1\ar[r]^{\gamma_1} & \mathcal X
}
$$
with $\mathcal Z_1$ smooth, and the rest of the diagram satisfying all of the conditions of the lemma at the generic point of $\mathcal D_1$.  
Proceeding inductively, one obtains a diagram as in \eqref{E:C:KawCovTrick}, with $\mathcal Z$ smooth, and satisfying all of the conditions of the lemma  at the generic points of the $\mathcal D_i$.  Note that one avoids some of the complications, compared to the construction \cite[Cor.~2.6]{Vmoduli} (which is via the construction in \cite[Lem.~2.5]{Vmoduli}, requiring normalizations of fibered products to ensure smoothness), since root stacks of smooth divisors on smooth stacks are smooth, and normal crossing divisors remain normal crossings after pull back by a root stack construction along one of the irreducible divisors.  
  Using purity of the branch locus (e.g., \cite[\href{https://stacks.math.columbia.edu/tag/0BMB}{Lem.~0BMB}]{stacks-project}), one obtains that $\tau'$ is \'etale.  The rest of the lemma follows by construction. 
\end{proof}

\subsection{Mild reduction}

Using \Cref{P:AK2000-Thm03} and the results in \cite{AK2000}, one gets a version of mild reduction as in \cite[\S 2]{VZ03} (\Cref{P:VZ03S2}).  We refer the reader to \cite[Def.~2.1]{VZ03} for the definition of a mild morphism of projective varieties.

\begin{pro}[Mild reduction]

\label{P:VZ03S2}
Let $f:\mathcal Y\to \mathcal X$ be a schematic projective morphism of smooth proper integral DM stacks over $\mathbb C$, with geometrically connected generic fiber, such that $\mathcal X$ has a projective coarse moduli space, and let $\mathbf \Delta\subseteq \mathcal X$ be a normal crossings divisor containing the discriminant of $f$.
  Then there exists a  diagram 
\begin{equation}\label{E:VZS2diag}
\xymatrix{
Z'  \ar[d]^{g'}&  Y''\ar[r]^\rho \ar[d]^{f''} \ar[l]_{\delta}&  Z \ar[r]^\sigma \ar[d]^g&  Y' \ar[r]^{\tau'} \ar[d]^{f'}& \widetilde {\mathcal Y}\ar[r]^{\widetilde \mu} \ar[d]^{\tilde f}&\mathcal Y \ar[d]^f\\
 X'\ar@{=}[r]&  X' \ar@{=}[r]&  X'\ar@{=}[r]&  X'\ar[r]^\tau& \widetilde {\mathcal X}\ar[r]^\mu& \mathcal X
}
\end{equation}
satisfying the following conditions:
\begin{enumerate}[label=(\roman*)]
\item  $\widetilde {\mathcal X}$ and $\widetilde {\mathcal Y}$ are smooth proper integral DM stacks  over $\mathbb C$ with projective coarse moduli spaces, $ X'$, $ Y''$, and $ Z$  are smooth projective varieties, $Y'$ is a normal projective variety, $Z'$ is a Gorenstein projective variety with at worst rational singularities, and  $\tilde f$ is projective.

\item $\tau$ is finite and flat and $ Y'$ is obtained as the normalization of the main component of $ X'\times_{\widetilde {\mathcal X}}\widetilde {\mathcal Y}$;

\item $\mu$, $\widetilde \mu$, $\rho$,  and $\delta$ are birational, $\mu$ is a blow-up, and $\sigma$ is a blowing up with center in the singular locus of $Y'$;

\item \label{P:VZ03S2DNC} there is an effective divisor  $\mathbf \Delta(X'/\widetilde {\mathcal X})\subseteq \widetilde {\mathcal X}$ containing the  discriminant of $\tau$, so that for  $\tilde {\mathbf \Delta}$ the support of $\mu^{-1}\mathbf \Delta$, one has that $\tilde {\mathbf \Delta} +\mathbf \Delta(X'/\widetilde {\mathcal X})$ and $\tilde f^*(\tilde {\mathbf \Delta} +\mathbf \Delta( X'/\widetilde {\mathcal X}))$ are snc divisors;

\item $g': Z'\to  X'$ is mild.
\end{enumerate}
\end{pro}

\begin{proof}
The argument is identical to that in \cite[\S 2]{VZ03}, except one uses  \Cref{P:AK2000-Thm03} to replace \cite[Thm.~0.3]{AK2000}, and  \cite[Thm.~A]{rydhComact} to replace  Raynaud--Gruson flattening (see also \cite[p.~721]{WW23}).  

For clarity, we include the details; diagram \eqref{E:VZ03S2const}, below,  will be useful for reference in what follows. The first step of the construction in \cite[\S 2]{VZ03} is to obtain a weak semi-stable reduction; we do this using  \Cref{P:AK2000-Thm03}, and obtain the weak semi-stable reduction
\begin{equation*}
\xymatrix{
Z_0' \ar[r] \ar[d] & \mathcal Y \ar[d]^f\\
X_0' \ar[r]& \mathcal X
}
\end{equation*}
We now apply 
Raynaud--Gruson flattening \cite[Thm.~A]{rydhComact} to the morphism $X_0'\to \mathcal X$.  This provides us with a blow-up $\mu':  {\mathcal X}_1\to \mathcal X$ such that the main component (strict transform) of the fibered product $\tau':X_1':=( \mathcal X_1\times_{\mathcal X}X_0')^{\sim} \to{\mathcal X}_1$ is flat; note that with $X_0'\to \mathcal X$ being generically finite, this implies that $\tau'$ is finite.  Let $\mathbf \Delta(X_1'/ {\mathcal X}_1)\subseteq \mathcal X_1$ denote the discriminant locus of $X_1'\to  {\mathcal X}_1$, and let $\mathcal B_1:= {\mathcal X}_1-\mu'^{-1}(\mathcal X-\mathbf \Delta)$ be the boundary divisor in $\mathcal X_1$.   Let $\widetilde {\mathcal X}\to \mathcal X_1$ be a strong log resolution of the pair $(\mathcal X_1,\mathbf \Delta(X_1'/ {\mathcal X}_1)+\mathcal B_1)$; denote the resulting boundary divisor in $\widetilde {\mathcal X}$ by $\widetilde {\mathcal B}$.  Take $\widetilde X'$ to be the normalization of the main component of $\widetilde {\mathcal X}\times_{\mathcal X_1}X_1'$, so that now $\mathbf \Delta(\widetilde X'/\widetilde {\mathcal X})+\widetilde {\mathcal B}$ is normal crossings. 
By Kawamata's covering construction, \Cref{C:KawCovTrick}, there exists a smooth proper integral DM stack $\mathcal X^"$ over $\mathbb C$ with projective coarse moduli space, which is finite over $\widetilde X'$. 
Let  $X'\rightarrow \mathcal{X}^"$ be a finite flat cover by a smooth projective variety so that the boundary divisor stays snc (see \cite[proof of Thm. 7.1]{CMZpositivity}). 

This now explains the spaces and morphisms in the bottom of diagrams \eqref{E:VZS2diag} and  \eqref{E:VZ03S2const}:
\begin{equation}\label{E:VZ03S2const}
\xymatrix@R=.5em{ Z'\ar[dd]_{g'} \ar[rrrrrrr]& &&&&&&Z_0'\ar@{-}[d]_{\text{weak-ssr}}\ar[rd]&\\
                        &&&&\widetilde {\mathcal Y}\ar[dd]_<>(0.25){\tilde f} \ar[rrrr]&&&\ar[d]&\mathcal Y \ar[dd]^f\\
                X' \ar[r] \ar@/_.5pc/[rrrrd]_\tau& \mathcal X^"\ar[rr]^{\text{Kawamata}} &&\widetilde X'\ar@{-}[r] \ar[rd]^{\tau''}&\ar[r]&X_1'\ar[rd]^{\tau'}\ar[rr]&&X_0'\ar[rd] &\\
                        &&&&\widetilde {\mathcal X}\ar[rr]^{\text{log-res}} \ar@/_2pc/[rrrr]^\mu &&\mathcal X_1\ar[rr]^{\text{RGf-blow-up}}_{\mu'}&&\mathcal X\\
}
\end{equation}
Take $\widetilde {\mathcal Y}$ to be a strong log resolution of singularities of the main component of the fibered product $\widetilde {\mathcal X}\times_{\mathcal X}\mathcal Y$.  Take  $ Z'\to  X'$ to be  the pull back of $Z_0'\to X_0'$, which is mild (\cite[Lem. 2.2(ii)]{VZ03}).
 We have explained the vertical morphisms displayed in \eqref{E:VZS2diag}, as well as the morphisms $g'$, $\tilde f$, and $f$ in \eqref{E:VZ03S2const}.
Let $\sigma:Z\to Y'$ be a strong resolution of singularities.  Finally, let $ Y''$ be a strong resolution of the birational map $ Z'\dashrightarrow  Z$.  This explains the rest of the diagram \eqref{E:VZS2diag}, and the rest of the assertions of the proposition follow by construction.

Note that from the definition of a mild morphism, we  have that $Z'$ is integral with rational singularities, and that $g'$ is a Gorenstein morphism; since in addition $X'$ is smooth, we can conclude  that  $Z'$ is Gorenstein \cite[\href{https://stacks.math.columbia.edu/tag/0C11}{Tag
 0C11}]{stacks-project}.
\end{proof}

\section{Duality on DM stacks and some lemmas of Viehweg and Zuo}\label{S:mild_redux}

We will recall some results of \cite{nironi09} on duality for DM stacks, in order to generalize some results of Viehweg and Zuo \cite{Vkod1, VZ03} to DM stacks.  

\subsection{Duality for DM stacks}\label{S:DuDM}
The main result that we want
 is \cite[Thm.~1.16]{nironi09}: \emph{Let $f:\mathcal  X\to \mathcal Y$ be a separated quasi-compact morphism of DM  stacks.  Then on the bounded derived categories of quasi-coherent sheaves, the derived push forward $Rf_*:D^+(\mathcal X)\to D^+(\mathcal Y)$ has a right adjoint $f^!:D^+(\mathcal Y)\to D^+(\mathcal X)$.}

A key property we will use is \cite[Cor.~2.3]{nironi09} that 
duality for proper morphisms of DM stacks is \'etale local in
the sense that given a $2$-cartesian diagram of quasi-compact morphisms of stacks
$$
\xymatrix{
\mathcal X'\ar[d]_{f'} \ar[r]^{g'}& \mathcal X \ar[d]_f\\
\mathcal Y'\ar[r]^g& \mathcal Y
}
$$
where $f$ is proper and $g$ is \'etale and representable, then there is a canonical isomorphism of functors on the bounded derived categories
$$
g'^*f^! \to f'^!g^*.
$$
We will also use that for a separated quasi-compact representable \'etale morphism $g:\mathcal Y'\to \mathcal Y$ of DM stacks, there is a natural identification $g^!=g^*$.   

 For  a separated quasi-compact morphism of DM stacks $f:\mathcal  X\to \mathcal Y$, the relative dualizing complex is 
 $
 \omega^\bullet_{\mathcal X/\mathcal Y}:= f^!\mathcal O_{\mathcal Y}.
$
For $\mathcal Y=\operatorname{Spec}\mathbb C$, we simply denote this by
$\omega_{\mathcal X}^\bullet$.  Note that if $p:U\to \mathcal X$ is an \'etale presentation, then we have
$
p^*\omega_{\mathcal X}^\bullet =p^!f^!\mathcal O_{\operatorname{Spec}\mathbb C}=\omega_{U}^\bullet,
$
the usual dualizing complex for $U$ over $\mathbb C$. For a stack $\mathcal X$ that has Gorenstein (resp.~Cohen--Macaulay or CM) singularities, we define the canonical sheaf of $\mathcal X$  as  $$\omega_{\mathcal X}:=\omega_{\mathcal X}^\bullet[-\dim \mathcal X],$$ which is also quasi-isomorphic to a line bundle (resp.~sheaf).

Assume that  $\mathcal X$ is  $S_2$ and $G_1$ (Gorenstein in codimension $1$), and let $j:\mathcal U\hookrightarrow \mathcal X$ be an open substack such that $\mathcal U$ is $G_1$ and the complement of $\mathcal U$ is codimension at least $2$.  Then we define the canonical sheaf of $\mathcal X$ as (e.g.,  \cite[Def.~5.6]{kovacs13_stable}):
$$
 \omega_{\mathcal X}:= j_*\omega_{\mathcal U}.
 $$   This is independent of the choice of such an open substack, and note that if $\mathcal X$ is CM, then this agrees with the definition of $\omega_{\mathcal X}$ above (e.g., the argument for \cite[Lem.~5.7]{kovacs13_stable}).  As $\omega_{\mathcal U}$ is a line bundle and $\mathcal X$ is $S_2$, we have that $\omega_{\mathcal X}$ is reflexive.  
Note also that if $\Omega^1_{\mathcal X}$ is the quotient of a torsion free sheaf on $\mathcal X$ (e.g., $\mathcal X$ is a closed substack of a smooth substack, so that the conormal sequence defines such a surjection), then  
$\omega_{\mathcal X}\cong \det \Omega^1_{\mathcal X}$; here we are taking the determinant as in \cite[\S 1.5]{CMZslope_stability}.

Returning to the relative case where $f:\mathcal X\to \mathcal Y$ is a separated quasi-compact morphism of integral separated DM stacks of finite type over $\mathbb C$ with $\mathcal X$ CM and $\mathcal Y$ Gorenstein, we have that $\omega_{\mathcal X/\mathcal Y}^\bullet [\dim \mathcal Y-\dim \mathcal X]$ is quasi-isomorphic to a sheaf we call the relative canonical sheaf, and will denote by $\omega_{\mathcal X/\mathcal Y}$.  Moreover, 
\begin{equation}\label{E:omegaX/Y-CM-G}
\omega_{\mathcal X/\mathcal Y} \cong \omega_{\mathcal X}\otimes f^*\omega_{\mathcal Y}^{-1};
\end{equation}
this follows \'etale-locally from the case of schemes (e.g., \cite[Rem.~26 (vii)]{kleiman80rel-duality}). If, in the same situation, we assume instead that $\mathcal X$ is only $S_2$ and $G_1$, then we use \eqref{E:omegaX/Y-CM-G} to define the relative canonical sheaf.  In the case where $\mathcal Y=\operatorname{Spec}\mathbb C$, one has that $\omega_{\mathcal X/\mathcal Y}=\omega_{\mathcal X}$. 

 For $m>0$ we observe that  
\begin{equation*}
\omega_{\mathcal X/\mathcal Y}^{[m]}:=(\omega_{\mathcal X/\mathcal Y}^{\otimes m})^{\vee\vee}= ((\omega_{\mathcal X})^{\otimes m})^{\vee \vee} \otimes f^*\omega_{\mathcal Y}^{-m}= \omega_{\mathcal X}^{[m]}\otimes f^*\omega_{\mathcal Y}^{-m}
\end{equation*}
so that if $\mathcal V\hookrightarrow \mathcal Y$ is an open substack such that $j:\mathcal U:=f^{-1}\mathcal V\hookrightarrow \mathcal X$ is an open substack such that $\mathcal U$ is $G_1$ and the complement is codimension at least $2$, then 
\begin{equation}\label{E:RelCanRefl}
\omega_{\mathcal X/\mathcal Y}^{[m]}=j_*\omega_{\mathcal U/\mathcal V}^{\otimes m}.
\end{equation}

Given a $2$-commutative diagram of separated quasi-compact morphisms of DM stacks
\begin{equation}\label{E:rel-tr}
\xymatrix{
\mathcal X \ar[rr]^f \ar[rd]_g &&\mathcal Y \ar[ld]^h\\
&\mathcal S&
}
\end{equation}
the adjoint property of the functors $(Rf_*,f^!)$ induces a so-called trace map
\begin{equation}\label{E:rel-tr-df}
tr: Rf_*\omega_{\mathcal X/\mathcal S} \to \omega_{\mathcal Y/\mathcal S}.
\end{equation}

The following will be useful in what follows:

\begin{lem}\label{L:rat}
Let $\mathcal X$ be an integral separated DM stack of finite type over $\mathbb C$.  Then $\mathcal X$ has rational singularities if and only if $\mathcal X$ is CM and for any resolution of singularities $\mu:\mathcal X'\to \mathcal X$, one has that the trace map $R\mu_*\omega_{\mathcal X'}\to \omega_{\mathcal X}$ is an isomorphism.  
\end{lem}

\begin{proof}
This follows \'etale locally from the standard results for varieties (e.g., \cite[Thm.~5.10]{KM98})
\end{proof}

Note that in the situation of \Cref{L:rat}, we have an induced morphism
\begin{equation}\label{E:counit-mu}
\mu^*\omega_{\mathcal X}=\mu^*\mu_*\omega_{\mathcal X'}\longrightarrow \omega_{\mathcal X'}, 
\end{equation}
which is an isomorphism away from the exceptional divisor.

\begin{cor}\label{C:rat-tens}
In the situation of \Cref{L:rat} with $\mathcal X$ assumed to have rational singularities, if $\mathcal X$ is Gorenstein, then the induced morphism
\begin{equation}\label{E:rat-tens}
\mu^*\omega_{\mathcal X}\longrightarrow \omega_{\mathcal X'} 
\end{equation}
induces for each $k\ge 1$ an isomorphism
$$
\mu^*\omega_{\mathcal X}^{\otimes k} \stackrel{\sim}{\longrightarrow} \omega_{\mathcal X'}^{\otimes k}(-k\mathcal E)
$$
for an effective exceptional divisor $\mathcal E$. 
In particular, Gorenstein and rational implies canonical (take $k=1$).  
\end{cor}

\begin{proof}
As explained in \eqref{E:counit-mu}, the morphism of line bundles \eqref{E:rat-tens} is an isomorphism outside of the exceptional divisor. The case with $k=1$ follows. The other cases follow by taking tensor products. 
\end{proof}

From \Cref{L:rat} we have the following corollary in the relative situation:

\begin{cor}\label{C:rel-rat}
In the situation of diagram \eqref{E:rel-tr}, assume that $\mathcal S$ is integral and normal with Gorenstein singularities,   $\mathcal Y$ is integral normal and CM, and $f$ is a resolution of singularities.  If $\mathcal Y$ has rational singularities, then the trace morphism $Rf_*\omega_{\mathcal X/\mathcal S}\to \omega_{\mathcal Y/\mathcal S}$ is an isomorphism.
\end{cor}

\begin{proof}
We have $Rf_*\omega_{\mathcal X/\mathcal S}= Rf_*(\omega_{\mathcal X}\otimes g^*\omega_{\mathcal S}^{-1}) =Rf_*\omega_{\mathcal X}\otimes h^*\omega_{\mathcal S}^{-1}=\omega_{\mathcal Y}\otimes h^*\omega_{\mathcal S}^{-1}=\omega_{\mathcal Y/\mathcal S}$.
\end{proof}

In the situation of \Cref{C:rel-rat}, we have an induced morphism
$$
f^*\omega_{\mathcal Y/\mathcal S}=f^*f_*\omega_{\mathcal X/\mathcal S}
\longrightarrow \omega_{\mathcal X/\mathcal S},
$$
which is an isomorphism away from the exceptional divisor.

\begin{cor}\label{C:rel-rat-tens}
In the situation of \Cref{C:rel-rat} with $\mathcal Y$ assumed to have rational singularities,  if $\mathcal Y$ is Gorenstein, then the induced morphism
\begin{equation}\label{E:rel-rat-tens}
f^*\omega_{\mathcal Y/\mathcal S}
\longrightarrow \omega_{\mathcal X/\mathcal S},
\end{equation}
induces for each $k\ge 1$ an isomorphism
$$
f^*\omega_{\mathcal Y/\mathcal S}^{\otimes k} \stackrel{\sim}{\longrightarrow} \omega_{\mathcal X/\mathcal S}^{\otimes k}(-k\mathcal E)
$$
for some effective exceptional divisor $\mathcal E$.  
\end{cor}

\begin{proof}
We have that $\omega_{\mathcal Y/\mathcal S}=\omega_{\mathcal Y}\otimes h^*\omega_{\mathcal S}^{-1}$ is a line bundle, and similarly for $\omega_{\mathcal X/\mathcal S}$.  
Thus, the morphism \eqref{E:rel-rat-tens} is a morphism of line bundles. Proof follows as in \Cref{C:rat-tens}
\end{proof}

\begin{cor}\label{C:rel-pluri}
In the situation of \Cref{C:rel-rat-tens}, for each $k\ge 1$, there is a canonical isomorphism
$$
g_*\omega_{\mathcal X/\mathcal S}^{\otimes k}\stackrel{\sim}{\longrightarrow} h_*\omega_{\mathcal Y/\mathcal S}^{\otimes k}.
$$
\end{cor}

\begin{proof}
We have $g_*\omega_{\mathcal X/\mathcal S}^{\otimes k}=h_*f_*\omega_{\mathcal X/\mathcal S}^{\otimes k}=h_*f_*(f^*\omega_{\mathcal Y/\mathcal S}^{\otimes k} \otimes \mathcal O_{\mathcal X}(k\mathcal E))=h_*(\omega_{\mathcal Y/\mathcal S}^{\otimes k}\otimes \mathcal O_{\mathcal Y})$. 
\end{proof}

\subsection{Some applications of duality}
We consider applications of duality following \cite[\S 3]{Vkod1}.

\subsubsection{The setup}\label{S:Vsetup}
 The setup is as follows.  We let 
$$
f:\mathcal Y\longrightarrow \mathcal X 
$$
be a surjective projective schematic morphism between smooth integral separated DM stacks of finite type over $\mathbb C$.  We define open substacks $\mathcal X_0\subseteq \mathcal X_1\subseteq \mathcal X_2\subseteq \mathcal X$ by the following conditions on their $\mathbb C$-points:
\begin{align*}
\mathcal X_2(\mathbb C)&:=\{x\in \mathcal X(\mathbb C): f \text { is flat along } f^{-1}(x)\}\\
\mathcal X_1(\mathbb C)&:=\{x\in \mathcal X(\mathbb C): f^{-1}(x)  \text { is reduced and of dimension } \dim \mathcal Y-\dim \mathcal X\}\\
\mathcal X_0(\mathbb C)&:=\{x\in \mathcal X(\mathbb C): f^{-1}(x)  \text { is smooth}\}
\end{align*}
and $\mathcal Y_i:=f^{-1}(\mathcal X_i)$ and $f_i:=f|_{\mathcal Y_i}:\mathcal Y_i\to \mathcal X_i$.  Note that $\mathcal X_1$ is contained in $\mathcal X_2$ by the miracle of flatness, and is open in $\mathcal X_2$ by standard results (e.g., \cite[\href{https://stacks.math.columbia.edu/tag/0574}{\S 0574}]{stacks-project}).

We then let $\tau:\mathcal X'\to \mathcal X$ be a flat schematic 
projective morphism from a smooth integral separated DM stack of finite type over $\mathbb C$. Let $\mathcal Y_{\mathcal X'}:=\mathcal X'\times_{\mathcal X}\mathcal Y$ be the fibered product, and consider a commutative diagram
$$
\xymatrix{
\mathcal Y'  \ar[r]^\mu \ar[d]^{f'}& \mathcal Y_{\mathcal X'}^\nu \ar[r]^\nu\ar[d]^{f^\nu_{\mathcal X'}}& \mathcal Y_{\mathcal X'} \ar[r]^{\tau_{\mathcal Y}} \ar[d]^{f_{\mathcal X'}}& \mathcal Y \ar[d]^f\\
\mathcal X' \ar@{=}[r]&\mathcal X' \ar@{=}[r]& \mathcal X' \ar[r]^\tau &\mathcal X
}
$$
where $\nu$ is the normalization and $\mu$ is a resolution of singularities. 

\subsubsection{Pluricanonical forms in families}

With the setup in \S \ref{S:Vsetup}, the first result we will want is a generalization of \cite[Lem.~3.2]{Vkod1}:

\begin{lem}
\label{L:LV3.2}
Assume that $\mathcal Y^\nu_{\mathcal X'}$ has rational singularities.  Then for all $k\ge 1$ there is a natural inclusion 
$$
i_k: f'_*\omega_{\mathcal Y'/\mathcal X'}^{\otimes k}\longrightarrow \tau^*f_*\omega_{\mathcal Y/\mathcal X} ^{\otimes k}.
$$
Moreover, $i_k$ is an isomorphism over $\tau^{-1}(\mathcal X_1)$ and over the open substack $\mathcal U'\subseteq \mathcal X'$ where $\tau$ is smooth. 
\end{lem}

\begin{proof}
 \'Etale locally, we obtain that $\mathcal Y_{\mathcal X'}$ is Gorenstein (see \cite[p.335]{Vkod1}).  
As $\tau$ is flat, there is a canonical isomorphism $\tau^*f_*\omega_{\mathcal Y/\mathcal X}^{\otimes k}\stackrel{\sim}{\to} f_{\mathcal X',*}\tau_{\mathcal Y}^*\omega_{\mathcal Y/\mathcal X}^{\otimes k}$. There is also a canonical morphism $\tau_{\mathcal Y}^*\omega^\bullet_{\mathcal Y/\mathcal X}=\tau_{\mathcal Y}^*f^!\mathcal O_{\mathcal X}\to f_{\mathcal X'}^!\tau^*\mathcal O_{\mathcal X}= f_{\mathcal X'}^!\mathcal O_{\mathcal X'}= \omega^\bullet_{\mathcal Y_{\mathcal X'}/\mathcal X' }$ \cite[(2.4)]{nironi09}.  Considering this morphism \'etale-locally, and reducing to  the case of varieties, one can confirm with standard arguments that the composition  $\tau_{\mathcal Y}^*\omega^\bullet_{\mathcal Y/\mathcal X}\stackrel{\sim}{\to } \omega^\bullet_{\mathcal Y_{\mathcal X'}/\mathcal X'}$ is an isomorphism (e.g., \cite[p.335]{Vkod1}).  Since pull back commutes with tensor products, we have $\tau_{\mathcal Y}^*\omega_{\mathcal Y/\mathcal X}^{\otimes k}\stackrel{\sim}{\to } \omega_{\mathcal Y_{\mathcal X'}/\mathcal X'}^{\otimes k}$. 
 In summary, we have canonical isomorphisms
$$
\tau^*f_*\omega_{\mathcal Y/\mathcal X}^{\otimes k}\stackrel{\sim}{\longrightarrow} f_{\mathcal X',*}\omega_{\mathcal Y_{\mathcal X'}/\mathcal X'}^{\otimes k}.
$$

Let $j:\mathcal V\hookrightarrow\mathcal Y^\nu_{\mathcal X'}$ be the open substack on which $\mu$ is an isomorphism; the complement has codimension at least $2.$ We obtain a canonical morphism (see \eqref{E:RelCanRefl} for the reflexive hull)
\[ \alpha_k: \mu_*(\omega_{\mathcal Y'/\mathcal X'}^{\otimes k} ) \hookrightarrow ( \mu_*\omega_{\mathcal Y'/\mathcal X'}^{\otimes k})^{\vee\vee} = \omega_{\mathcal Y^\nu_{\mathcal X'}/\mathcal X'}^{[k]}.
\]
 Applying $f^\nu_{\mathcal X',*}$, we obtain an inclusion
\[ f'_*\omega_{\mathcal Y'/\mathcal X'}^{\otimes k} \hookrightarrow f^\nu_{\mathcal X',*} \omega_{\mathcal Y^\nu_{\mathcal X'}/\mathcal X'}^{[k]}. \]

The next step is to construct natural morphisms
\begin{equation}\label{E:LV3.2}
 f^\nu_{\mathcal X',*}\omega_{\mathcal Y^\nu_{\mathcal X'}/\mathcal X'}^{[k]} \longrightarrow  f_{\mathcal X',*}\omega_{\mathcal Y_{\mathcal X'}/\mathcal X'}^{\otimes k}.
\end{equation}
Once we construct these morphism in \eqref{E:LV3.2}, the composition of the morphisms above defines $i_k$:
$$
\xymatrix{
i_k: 
f'_*\omega_{\mathcal Y'/\mathcal X'}^{\otimes k}\ar@{^(->}[r] & 
f^\nu_{\mathcal X',*}\omega_{\mathcal Y^\nu_{\mathcal X'}/\mathcal X'}^{[k]} \ar[r]&  f_{\mathcal X',*}\omega_{\mathcal Y_{\mathcal X'}/\mathcal X'}^{\otimes k}\ar[r]^<>(0.5){\sim}&  \tau^*f_*\omega_{\mathcal Y/\mathcal X}^{\otimes k},
}
$$
and all of the other assertions of \Cref{L:LV3.2}  can be checked \'etale-locally, and therefore follow from \cite[Lem.~3.2]{Vkod1}.

To construct the morphism \eqref{E:LV3.2}, we start by constructing a commutative diagram
$$
\xymatrix{
\nu^*\nu_*\omega_{\mathcal Y^\nu_{\mathcal X'}/\mathcal X'} \ar[rr]^{\nu^*tr} \ar[rd]_\epsilon& &\nu^*\omega_{\mathcal Y_{\mathcal X'}/\mathcal X'}\\
& \omega_{\mathcal Y^\nu_{\mathcal X'}/\mathcal X'}  \ar@{-->}[ru]_\beta.
}
$$
In the above,  $tr: \nu_*\omega_{\mathcal Y^\nu_{\mathcal X'}/\mathcal X'} \to \omega_{\mathcal Y_{\mathcal X'}/\mathcal X'}$ is the trace \eqref{E:rel-tr-df}, where we are using that $\nu$ is finite to make the identification $R\nu_*=\nu_*$, and the morphism $\epsilon: \nu^*\nu_*\omega_{\mathcal Y^\nu_{\mathcal X'}/\mathcal X'} \to \omega_{\mathcal Y^\nu_{\mathcal X'}/\mathcal X'} $ is the co-unit of adjunction.  
 The claim is that the  pull back of the trace factors through the co-unit of adjunction, defining the dashed arrow $\beta$.
For this, observe first that the co-unit of adjunction is surjective, since $\nu$ is finite, and therefore affine.  Then,  since the normalization $\nu$ is an isomorphism away from the conductor, the above diagram consists of isomorphisms on the complement of the conductor.  In particular, the kernel of $\epsilon$ is torsion.  Since $\omega_{\mathcal Y_{\mathcal X'}/\mathcal X'}$ is a line bundle, one has that its pull back by $\nu$ is also a line bundle, and so the kernel of $\epsilon$ is contained in the kernel of $\nu^*tr$, completing the claim.  Note also that as $\omega_{\mathcal Y^\nu_{\mathcal X'}/\mathcal X'}$ is torsion-free, and $\beta$ is generically an isomorphism, we have that $\beta$ is injective.

On $\mathcal V\subseteq \mathcal Y^\nu_{\mathcal X'}$, the morphism $\beta$ gives
\[ 1\otimes\beta^{\otimes(k-1)}: \omega_{\mathcal Y^\nu_{\mathcal X'}/\mathcal X'}^{\otimes k}
|_{\mathcal V} \rightarrow (\omega_{\mathcal Y^\nu_{\mathcal X'}/\mathcal X'}
\otimes \nu^*\omega_{\mathcal Y_{\mathcal X'}/\mathcal X'}^{\otimes(k-1)}
)|_{\mathcal V}.
\]
Thus, using the reflexivity of the sheaves on the right, we obtain
\[
\beta_k:
\omega_{\mathcal Y^\nu_{\mathcal X'}/\mathcal X'}^{[k]}
\longrightarrow
\omega_{\mathcal Y^\nu_{\mathcal X'}/\mathcal X'}
\otimes
\nu^*\omega_{\mathcal Y_{\mathcal X'}/\mathcal X'}^{\otimes(k-1)}.
\]

We can then push forward $\beta_{k}$ to $\mathcal X'$, or more precisely, we do the following:
\begin{align*}
 f^\nu_{\mathcal X',*}(\omega_{\mathcal Y^\nu_{\mathcal X'}/\mathcal X'}^{[k]})
& \stackrel{ f^\nu_{\mathcal X',*}(\beta_{k})}{\longrightarrow} f^\nu_{\mathcal X',*} (\omega_{\mathcal Y^\nu_{\mathcal X'}/\mathcal X'}\otimes \nu^*\omega_{\mathcal Y_{\mathcal X'}/\mathcal X'}^{\otimes k-1}) \\
& = f_{\mathcal X',*} \nu_* (\omega_{\mathcal Y^\nu_{\mathcal X'}/\mathcal X'}\otimes \nu^*\omega_{\mathcal Y_{\mathcal X'}/\mathcal X'}^{\otimes k-1})\\
& = f_{\mathcal X',*} (\nu_* \omega_{\mathcal Y^\nu_{\mathcal X'}/\mathcal X'}\otimes \omega_{\mathcal Y_{\mathcal X'}/\mathcal X'}^{\otimes k-1})\\ 
&  \stackrel{tr\otimes 1}{\longrightarrow}
f_{\mathcal X',*} ( \omega_{\mathcal Y_{\mathcal X'}/\mathcal X'}\otimes \omega_{\mathcal Y_{\mathcal X'}/\mathcal X'}^{\otimes k-1})
 \end{align*}
This defines the morphism \eqref{E:LV3.2}. We remark that the construction of the morphism did not use the assumption that $\mathcal Y^\nu_{\mathcal X'}$ have rational singularities; that assumption enters in the statement that we have an isomorphism over $\tau^{-1}(\mathcal X_1)$.
\end{proof}

\subsubsection{The fibered product trick}
In the set-up of \S \ref{S:Vsetup}, assume  that 
 $f : \mathcal Y \to  \mathcal X$ is a  surjective schematic projective morphism between smooth separated integral DM stacks of finite type over $\mathbb C$ admitting projective coarse moduli spaces,  and $f_*\mathcal O_{\mathcal Y}=\mathcal O_{\mathcal X}$. 
Define 
$$
\mathcal Y^{s}:=\mathcal Y\times_{\mathcal X}\mathcal Y \times_{\mathcal X}\cdots \mathcal Y\times_{\mathcal X}\mathcal Y \ \ \text{($s$ times)} 
$$
and let 
$
\mu^s:\mathcal Y^{(s)}\longrightarrow \mathcal Y^s
$
be a resolution of singularities.  Let $f^s:\mathcal Y^s\to \mathcal X$ be the induced morphism, and set $f^{(s)}:=f^s\circ \mu^s:\mathcal Y^{(s)}\to \mathcal X.$
Assume (using the notation introduced
in \S \ref{S:Vsetup}) that 
$\mathcal X-\mathcal X_0$ and $\mathcal Y-\mathcal Y_0$ 
are normal crossing divisors, and
that $\operatorname{codim}( \mathcal X-\mathcal X_1)\ge 2$.

With this set-up, we can state the following generalization of \cite[Lem.~3.5]{Vkod1}:

\begin{lem}

\label{L:LV3.5}
There exists an open substack $\mathcal U\subseteq \mathcal X$ such that $\operatorname{codim}(\mathcal X-\mathcal U)\ge 2$ so that over $\mathcal U$ there is a  natural isomorphism
$$
f^{(s)}_*\omega_{\mathcal Y^{(s)}/\mathcal X}^{\otimes k}|_{\mathcal U}  \stackrel{\sim}{\longrightarrow} \bigotimes_{i=1}^sf_*\omega_{\mathcal Y/\mathcal X}^{\otimes k}|_{\mathcal U}.
$$
\end{lem}

\begin{proof}
We may replace $\mathcal X$  by $\mathcal X_1$ and $\mathcal Y$  by $\mathcal Y_1$ (using the notation in \S \ref{S:Vsetup}). Hence we may
assume that $f$ is flat and has reduced fibers, and that the locus in  $\mathcal Y$  where
$f$ is not smooth is of codimension at least two. Hence the
singular locus of $\mathcal Y^s$  is also of codimension at least two.  The first claim is that $\mathcal Y^s$ is normal and Gorenstein. This follows \'etale-locally from the case of varieties \cite[p.337]{Vkod1}.  Consequently, $\omega_{\mathcal Y^s/\mathcal X}= \omega_{\mathcal Y^s}\otimes f^{s,*}\omega_{\mathcal X}^{-1}$ is a line bundle, and one can check that $\omega_{\mathcal Y^s/\mathcal X}=\bigotimes_{i=1}^s\operatorname{pr}_i^*\omega_{\mathcal Y/\mathcal X}$, where $\operatorname{pr}_i$ denotes the $i$-th projection. 

If $\mathcal Y^s$ has rational singularities, then, from \Cref{C:rel-pluri}, we have a natural isomorphism:
$$
f^{(s)}_*\omega_{\mathcal Y^{(s)}/\mathcal X}^{\otimes k} \longrightarrow f^{s}_*\omega_{\mathcal Y^{s}/\mathcal X}^{\otimes k}= f^s_*(\bigotimes_{i=1}^s\operatorname{pr}_i^*\omega_{\mathcal Y/\mathcal X}^{\otimes k})=\bigotimes_{i=1}^sf_*\omega_{\mathcal Y/\mathcal X}^{\otimes k},
$$
where the last equality is the standard Kunneth type statement.
The rationality of singularities follows \'etale-locally from the case of varieties (see \cite[Lem.~3.5 and Lem.~3.6]{Vkod1}). Consequently, after possibly replacing $\mathcal X$ with an open substack with complement of codimension still at least two, one may assume that $\mathcal Y^s$ has rational singularities, completing the proof.
\end{proof}

\subsection{A further application of duality following Viehweg--Zuo}
We start with the following generalization of  \cite[Cor.~2.4]{VZ03}:

\begin{lem}

\label{L:VZ03Cor2.4}
Let $f:\mathcal Y\to \mathcal X$ be a schematic projective morphism of smooth proper integral DM stacks over $\mathbb C$, with geometrically connected fibers, such that $\mathcal X$ has a projective coarse moduli space, and $\mathbf \Delta$ is a normal crossings divisor containing the discriminant of $f$.

 Then there exists a  diagram 
\begin{equation}\label{E:VZ03Cor2.4}
\xymatrix{
Z'  \ar[d]^{g'}&  Y''\ar[r]^\rho \ar[d]^{f''} \ar[l]_{\delta}&  Z \ar[r]^\sigma \ar[d]^g&  Y' \ar[r]^{\tau'} \ar[d]^{f'}& \widetilde {\mathcal Y}\ar[r]^{\widetilde \mu} \ar[d]^{\tilde f}&\mathcal Y \ar[d]^f\\
 X'\ar@{=}[r]&  X' \ar@{=}[r]&  X'\ar@{=}[r]&  X'\ar[r]^\tau& \widetilde {\mathcal X}\ar[r]^\mu& \mathcal X
}
\end{equation}
satisfying the conditions in the mild reduction \Cref{P:VZ03S2}.  We have, moreover:

\begin{enumerate}[label=(\roman*)]

\item  \label{L:VZ03Cor2.4i} $Y'$ has rational singularities;

\item \label{L:VZ03Cor2.4ii} for all $k \ge 1$ there exist isomorphisms
$$
\xymatrix{
g'_*\omega_{ Z'/ X'}^{\otimes k} \ar[r]^<>(0.5)\sim& f''_*\omega_{ Y''/ X'}^{\otimes k} & g_*\omega_{ Z/ X'}^{\otimes k} \ar[l]_<>(0.5)\sim;
}
$$
in particular, $g_*\omega_{ Z/ X'}^{\otimes k}$ is a reflexive sheaf;

\item \label{L:VZ03Cor2.4iii} for all $k\ge 1$ there exists a canonical  inclusion
\begin{equation}\label{E:VZ03Cor2.4iii}
\xymatrix{
g_*\omega_{ Z/ X'}^{\otimes k} \ar@{^(->}[r]^\iota & \tau^*\tilde f_*\omega^{\otimes k}_{\widetilde {\mathcal Y}/\widetilde {\mathcal X}}
}
\end{equation}
that is an isomorphism outside of $\tau^{-1}\mu^{-1}(\mathbf{\Delta}_{f})$.
\item \label{L:VZ03Cor2.4iv} For all $k \geq 1$, there exists some $N_k$ and an
invertible sheaf $\lambda_{k}$ on $\widetilde{\mathcal{X}}$ with
$$
\tau^* \lambda_{k} \simeq \operatorname{det} (g_*
\omega^{\otimes k}_{ Z/ X'})^{\otimes N_{k}}.
$$

\end{enumerate}

\end{lem}

\begin{proof}

\ref{L:VZ03Cor2.4i} This follows \'etale-locally from the case of varieties (see \cite[Cor.~2.4]{VZ03}).  

\ref{L:VZ03Cor2.4ii} This follows from the fact that $ Z'$ and $ Z$ have rational singularities (in fact $ Z$ is smooth), and  \Cref{C:rel-pluri}, showing that the sheaves on the left and right in \ref{L:VZ03Cor2.4ii} are isomorphic to the sheaf in the middle.  The reflexivity of all the sheaves then follows from the fact that $g'_*\omega_{ Z'/ X'}^{\otimes k}$ is reflexive, via \cite[Lem. 2.2(iii)]{VZ03}.  

\ref{L:VZ03Cor2.4iii} Both of the statements follow from \Cref{L:LV3.2}. For the second one note that $\X_1$ in \ref{S:Vsetup} contains the locus over which $f$ is smooth. 

\ref{L:VZ03Cor2.4iv} This follows the same argument as in \cite[Cor.2.4(ix)]{VZ03}. We include the argument here for convenience, and expand on the presentation there to make it clear that the arguments there hold in the situation here of DM stacks. 
To begin, let $B_{k}:=\det (\iota)$ on $X'$ denote the (effective) zero divisor of the determinant of the inclusion \eqref{E:VZ03Cor2.4iii}, so that 
$$
\det (g'_* \omega^{\otimes k}_{ Z'/ X'} ) \otimes \mathcal O_{ X'} ( B_{k}) =
\tau^* \det (\tilde{f}_* \omega^{\otimes k}_{\widetilde{\mathcal{Y}}/\widetilde{\mathcal{X}}});
$$
note that from \ref{L:VZ03Cor2.4iii}, the support of $B_k$ is contained in the support of $\tau^{-1}\mu^{-1}(\mathbf \Delta_f)$.  
To prove \ref{L:VZ03Cor2.4iv}, it suffices to show  that for all $k$ there is some sufficiently high multiple of $B_k$  that is the pull back via $\tau$ of a $\mathbb Q$-divisor $\widetilde {\mathcal B}_k$ on $\widetilde{\mathcal{X}}$.  
The basic strategy is to reduce to computing multiplicities after restriction to test curves $C\to \widetilde {\mathcal X}$; the details are somewhat lengthy, and so we break the argument into several steps.

\subsubsection*{Step 1: Reduction to  testing multiplicities of divisors over curves}
In order to show that a multiple of $B_{k}$ is the pull back via $\tau$ of a $\Q$-divisor $\widetilde {\mathcal B}_k$ on $\widetilde{\mathcal{X}}$, and thereby complete the proof, it suffices to show that if two irreducible components $B'$ and $B''$ of $B_k$ map to the same divisor $\widetilde {\mathcal B}$  in $\widetilde {\mathcal X}$, then the multiplicities of the two components $B'$ and $B''$ in $B_k$, give the same multiple of  $\tau^{-1}\widetilde {\mathcal B}$ on those components.
To see that this suffices, note  that $\widetilde{\mathcal X}$ is a smooth DM stack and ${X}'$ is a scheme; identifying  
$\operatorname{CH}^1(\widetilde {\mathcal{X}})_{\mathbb{Q}}=\operatorname{CH}^1(\widetilde X)_{\mathbb{Q}}$, one is then reduced to the case of normal $\mathbb Q$-factorial varieties, where this is standard.

We would like to further reduce to checking this condition over curves.    To this end, it is clear that it suffices to check the condition in the previous paragraph,  regarding  multiplicities of the components $B'$ and $B''$, at the generic points of those divisors.  Moreover, if  $C\to \widetilde {\mathcal X}$ is a general morphism from a smooth quasi-projective  curve $C$,  whose image in $\widetilde{\X}$ intersects
$\widetilde{\mathcal{B}}$ in exactly one point, which is a  general point $\tilde x$ of $\widetilde {\mathcal B}$, then it suffices to show that  the multiplicities of the two restrictions  
\begin{equation}\label{E:VZ03Cor2.4S1}
B'|_{C}:=C \times _{\widetilde {\mathcal X}}B' \quad \text{and} \quad B''|_{C}:=C\times _{\widetilde {\mathcal X}}B''
\end{equation}
   are the same at the generic points.  
In the following steps we will show this.

\subsubsection*{Step 2: Reduction of diagram \eqref{E:VZ03Cor2.4} to a general curve $C\to \widetilde {\mathcal X}$}
In order to complete \emph{Step 1}, we will reduce the question to  families over curves, by taking base changes  over sufficiently general morphisms $C\to \widetilde {\mathcal X}$ from smooth quasi-projective curves.   In order to do this, we will need to inductively use \Cref{L:VZ03Cor2.4}\ref{L:VZ03Cor2.4i}--\ref{L:VZ03Cor2.4iii}, and in particular, we will need to show that for sufficiently general morphisms   $C\to \widetilde {\mathcal X}$ from smooth quasi-projective curves $C$, we can obtain a version of diagram  \eqref{E:VZ03Cor2.4}  over   $C$ satisfying  the conditions in \Cref{P:VZ03S2}.  

More precisely, we claim the following.  Fix a  component
$\widetilde{\mathcal{B}}$ of $\mu^{-1}(\mathbf{\Delta}_{{f}})$,  let $\tilde x$ be a general point of 
$\widetilde{\mathcal{B}}$.
For a general morphism $C\to \widetilde {\mathcal X}$ from a smooth  quasi-projective   curve,  whose image in $\widetilde{\X}$ intersects
$\widetilde{\mathcal{B}}$ in a general point $\tilde x$,   
  we have  a diagram 
\begin{equation}\label{E:VZ03Cor2.4C}
\xymatrix{
Z'_{C'}  \ar[d]^{g'_{C'}}&  Y''_{C'}\ar[r]^{\rho_{C'}} \ar[d]^{f''_{C'}} \ar[l]_{\delta_{C'}}&  Z_{C'} \ar[r]^{\sigma_{C'}} \ar[d]^{g_{C'}}&  Y'_{C'} \ar[r]^{\tau'_{C}} \ar[d]^{f'_{C'}}& \widetilde { Y}_C  \ar[d]^{\tilde f_C}\\
 C'\ar@{=}[r]&  C' \ar@{=}[r]&  C'\ar@{=}[r]&  C'\ar[r]^{\tau_C}&C 
}
\end{equation}
satisfying the conditions in \Cref{P:VZ03S2}, as well as the additional conditions  that the base change $C':=C\times_{\widetilde {\mathcal X}} X'$ is a smooth quasi-projective curve, $\tau_C:C'\to C$ is the base change of $\tau$, and the spaces $\widetilde Y_C:=C\times_{\widetilde {\mathcal X}}\widetilde {\mathcal Y}$ and $Z':=C'\times_{X'}Z'$ are obtained by base change.  For the other spaces, we take $Y'_{C'}$ to be the normalization of the main component of $C'\times_C\widetilde {Y}_C$, we take $\sigma_{C'}:Z_{C'}\to Y'_{C'}$ to be a strong resolution of singularities, and we take $Y''_{C'}$, and $\delta_{C'}$ and $\rho_{C'}$ to be a smooth resolution of the birational morphism $Z'_{C'}\dashrightarrow Z_C$.   The other morphisms in the diagram are defined in the obvious way. 

 We now explain how to establish the claim.  
 Take a finite flat cover  $V\to \widetilde{\X}$ from a smooth projective variety $V$, 
 which is \'etale over our chosen point $\tilde x$ (see \S \ref{S:DM-intro}).  Since we are only interested in the end in a quasi-projective curve, let us replace $V$ with the quasi-projective variety over which $V\to \widetilde {\mathcal X}$ is \'etale.  Consider the fibered product $\widetilde{{ Y}}_{V} $, which is smooth, as $V\to \widetilde {\mathcal X}$ is \'etale.   Pick $H$ an effective very ample divisor on $\widetilde X$,  the projective coarse moduli space of $\widetilde{\mathcal X}$.  Let $H_1$ be a general member of the linear system on $V$ obtained by the pull back of $H$;  by the base-point free Bertini theorem in characteristic $0$, it follows that $H_1$ is smooth.  For the same reason, the fibered product $\widetilde Y_{H_1}:= \widetilde Y_V|_{H_1}$ will be smooth.  Iterating, we obtain a morphism $C\to \widetilde {\mathcal X}$ from a quasi-projective curve, such that $\widetilde Y_C:=C\times_{\widetilde {\mathcal X}} \widetilde {\mathcal Y}$ is smooth.  

Next, while $\tau:X'\to \widetilde {\mathcal X}$ may be branched over $\widetilde {\mathcal B}$, nevertheless, by taking the hyperplanes general above, we can assume that the base change $\tau_{C}:C'=C\times_{ \widetilde {\mathcal X}} X'\to C$ is a cover of smooth quasi-projective curves.  The base change $Z'_{C'}=C'\times_X'Z'$ is mild by \cite[Def.~2.1]{VZ03}. 
   The rest of the morphisms in \Cref{P:VZ03S2} are obtained as described here, completing the claim.

\subsubsection*{Step 3: Reduction to showing $\det(\iota_C)$ is the pull back of a divisor on $C$}

We now apply  \Cref{L:VZ03Cor2.4}\ref{L:VZ03Cor2.4iii} with
$\widetilde{\mathcal{X}}$ replaced by $C$.  In fact, as we are working now in the case of varieties, we are free to employ \cite[Cor.~2.4]{VZ03}, and we note here that, as in \cite[Cor.~2.4(ix)]{VZ03},  the conclusions hold even though $C$ is only assumed to be quasi-projective.   In particular, one obtains a natural inclusion
\begin{equation}\label{incl2}
\iota_C: g'_{C'*} \omega^{\otimes k}_{{Z}'_{C'}/C'} \rightarrow
\tau_C^* (\tilde{f}_{C*} \omega^{\otimes k}_{\widetilde{{Y}}_C/C}).
\end{equation}
Hence, we have a commutative diagram
\begin{equation}\label{D:incl2}
\begin{tikzcd}
g'_{C'*} \omega^{\otimes k}_{{Z}'_{C'}/C'} \arrow[r, "\iota_C", hook]                           & \tau_C^* (\tilde{f}_{C*} \omega^{\otimes k}_{\widetilde{{Y}}_C/C})                                                                                                                            \\
g'_{*} \omega^{\otimes k}_{{Z}'/{X}'} |_{C'} \arrow[u, "\alpha"] \arrow[r, "\iota|_{C'}", hook] & (\tau^* (\tilde{f}_{*} \omega^{\otimes k}_{\widetilde{\mathcal{Y}}/\mathcal{X}}))|_{C'}= \tau_C^* (\tilde{f}_{*} \omega^{\otimes k}_{\widetilde{\mathcal{Y}}/\mathcal{X}}|_C) \arrow[u, "\beta"]
\end{tikzcd}
\end{equation}
where $\alpha$ and $\beta$ are induced by the canonical morphisms from the pull back of the push forward to the push forward of the pull back. 

We claim that  up to a divisor pulled back by $\tau_C$, the zero divisor of $\det(\iota_C)$ on $C'$ is equal $\det(\iota )|_{C'}=B_k|_{C'}$,  so that 
 in order to complete the proof of \ref{L:VZ03Cor2.4iv}, we will then just have to verify
that $\det(\iota_C)$ is the pull back of a $\Q$-divisor on $C$.
To prove the claim we will show that $\alpha$ is an isomorphism. This suffices, since $\beta$ is the pull back of some morphism under $\tau_C$. As such, the difference in determinants for $\beta$ is the pull back of some divisor by $\tau_C$; 
thus, up to a divisor pulled back by $\tau_C$, the zero divisor of $\det(\iota_C)$ on $C'$ is $\det(\iota)|_{C'}$.

We now explain why  $\alpha$ is an isomorphism, provided we chose the divisors  $H_i\subseteq V$ in the construction of $C$ in \emph{Step 2} to be
very general.  In other words, we will show $\alpha$ is an isomorphism  inductively on restriction to the $H_i$.  To this end, by picking $H=H_1$ very generally we can ensure that $\mathcal{T}or_1^{\mathcal O_{X'}}(R^1g'_*\omega^{\otimes m}_{Z'/X'},\mathcal{O}_{H})=0$, because to achieve this, one only has to make sure that $H$ does not contain any of the associated primes of the countably many sheaves $\big(R^1g'_*\omega^{\otimes m}_{Z'/X'}\big)$ for $m\ge 1$. Denote by $Z'_H:=g'^{-1}(H)$.
On the one hand, we have the short exact sequence 
 $$
\xymatrix{
0\ar[r] & g'_*(\omega^{\otimes m}_{Z'/X'})\otimes_{\mathcal O_{X'}}\mathcal{O}_{X'}(-H) \ar[r]^<>(0.5){\cdot h} & g'_*\omega^{\otimes m}_{Z'/X'} \arrow[r] & g'_*(\omega^{\otimes m}_{Z'/X'})\otimes_{\mathcal O_{X'}}\mathcal{O}_{H} \ar[r]&0
}
$$
obtained by tensoring the standard exact sequence for $H$ by $g'_*(\omega^{\otimes m}_{Z'/X'})\otimes_{\mathcal O_{X'}}(-)$.  We obtain the left exactness as follows.  By virtue of the fact that $g'_*(\omega^{\otimes m}_{Z'/X'})$ is reflexive \cite[Lem. 2.2(iii)]{VZ03}, and therefore locally free outside of codimension $2$, the sequence is left exact outside of codimension $2$.  In particular, the kernel of the multiplication map $\cdot h$ is torsion, and being a subsheaf of the torsion-free sheaf $g'_*(\omega^{\otimes m}_{Z'/X'})\otimes_{\mathcal O_{X'}}\mathcal{O}_{X'}(-H)$, the kernel is trivial. 

 On the other hand,  using that $g'$ is flat, we obtain by pull back  a short exact sequence   $$0 \ra g'^*\mathcal{O}_{X'}(-H)\ra g'^*\mathcal{O}_{X'} \ra g'^*\mathcal{O}_{H}= \mathcal O_{Z'_{H}}\ra 0,$$
 and arguing as above, we obtain a short exact sequence 
 $$
\xymatrix{
0  \ar[r] & \omega^{\otimes m}_{Z'/X'}\otimes g'^*\mathcal{O}_{X'}(-H) \ar[r]^<>(0.5){} & \omega^{\otimes m}_{Z'/X'} \ar[r] & \omega^{\otimes m}_{Z'/X'}\otimes\mathcal{O}_{Z'_H} \ar[r]& 0
}
$$
so that pushing forward, and using the projection formula, we obtain 
the  long exact sequence 
\[
\xymatrix{
0  \ar[r] & g'_*(\omega^{\otimes m}_{Z'/X'}\otimes_{\mathcal O_{Z'}} g'^*\mathcal{O}_{X'}(-H)) \ar[r]^<>(0.5){\cdot h} & g'_*\omega^{\otimes m}_{Z'/X'} \ar[r] & g'_*(\omega^{\otimes m}_{Z'/X'}\otimes_{\mathcal O_{Z'}}\mathcal{O}_{Z'_H}) \ar[lld] \\
             & (R^1g'_*\omega^{\otimes m}_{Z'/X'})\otimes_{\mathcal O_{X'}} \mathcal{O}_{X'}(-H) \ar[r]^<>(0.5){\cdot h}  & R^1g'_*\omega^{\otimes m}_{Z'/X'}       \ar[r] &      R^1g'_*(\omega^{\otimes m}_{Z'/X'}\otimes_{\mathcal O_{Z'}}\mathcal{O}_{Z'_H})                                                                   
}
\]
Thus, since  $\mathcal{T}or_1^{\mathcal O_{X'}}(R^1g'_*\omega^{\otimes m}_{Z'/X'},\mathcal{O}_{H})=0$, 
the kernel of the multiplication map in the second row, $(R^1g'_*\omega^{\otimes m}_{Z'/X'})\otimes_{\mathcal O_{X'}} \mathcal{O}_{X'}(-H) \stackrel{\cdot h} \to R^1g'_*\omega^{\otimes m}_{Z'/X'}$, is trivial, and so
we obtain exactness of
 $$
\xymatrix{
0\ar[r] & g'_*(\omega^{\otimes [m]}_{Z'/X'})\otimes_{\mathcal O_{X'}}\mathcal{O}_{X'}(-H) \ar[r]^<>(0.5){\cdot h} & g'_*\omega^{\otimes m}_{Z'/X'} \arrow[r] &g'_*(\omega^{\otimes m}_{Z'/X'}\otimes_{\mathcal O_{Z'}}\mathcal{O}_{Z'_H}) \ar[r]&0
}
$$
implying that 
\[g'_*(\omega^{\otimes m}_{Z'/X'})\otimes_{\mathcal O_{X'}}\mathcal{O}_{H}\cong  g'_*(\omega^{\otimes m}_{Z'/X'}\otimes\mathcal{O}_{Z'_H}).\]

To complete the proof of the claim,  we need to check Koll\'ar's condition, i.e.,
\begin{equation} \label{check-Kollar-condition}
\omega^{\otimes m}_{Z'/X'}\otimes\mathcal{O}_{Z'_H}\cong \omega^{\otimes m}_{Z'_H/H}.
\end{equation}
Recall from \Cref{P:VZ03S2} that $Z'$ is Gorenstein.
 For $m=1$, as the morphism $g'$, being mild, is flat and projective with CM fibers, then by \cite[2.68.2]{kollar_families}, \eqref{check-Kollar-condition} is an isomorphism. Consequently, for higher $m$, as $Z'$ is Gorenstein and mild morphisms are flat projective and the base is smooth (in particular Gorenstein), $\omega_{Z'/X'}$ is a line bundle,  and so one has that  \eqref{check-Kollar-condition} is an isomorphism from the case $m=1$. Inductively, we have shown that $\alpha$ is an isomorphism.

\subsubsection*{Step 4: Showing $\det(\iota_C)$ is the pull back of a divisor on $C$} 
 In order to show \ref{L:VZ03Cor2.4iv}, we just have to verify
that $\det (\iota_C)$ is the pull back of a $\Q$-divisor on $C$.

By \cite{KKMS} there exists a finite morphism $\kappa: C'' \to C$, totally ramified over $x\in C$ (and \'etale elsewhere; note that the curves are not assumed to be proper), such that 
$\widetilde{{Y}}_{C''}:= \widetilde{{Y}}_C \times_{C} C''$ has a semi-stable model $S \to C''$. Define $C'''$ to be the normalization of the main component of $C''\times_CC'$, define $S':=C'''\times_{C''}S$, define $Z'_{C'''}:=C'''\times_C'Z_{C'}$, and let $\widetilde Z$ be a resolution of singularities, giving the diagram below with the various morphisms among the spaces as marked: 
\[
\xymatrix@C=2em@R=1em{
& \widetilde {Z} \ar[ld]_{\eta_1} \ar[rd]^{\eta_2} & & & & & & \\
{S}' \ar[ddd]_{\tau_S} \ar[rrrd]_{pr_1}
& & Z'_{C'''} \ar[rd]^{g'_{C'''}} \ar[rr] \ar@{<-->}[ll]
& & Z'_{C'} \ar[rd]^{g'_{C'}} \ar[rr]
& & Z' \ar[rd]^{g'} & \\
& & & C''' \ar[rr]_<>(0.25){\kappa'} \ar[ddd]_{\tau_{C''}}
& \ar[u]& C'\ar[ddd]_{\tau_C} \ar[rr]
&\ar[u] & X' \ar[ddd]_{\tau} \\
& & 
&& Y''_{C'}  \ar[ru]\ar[d] \ar@{-}[u] \ar@{-}[r]
& \ar[r]& Y''  \ar[ru]\ar[d] \ar@{-}[u] & \\
S \ar[rrrd]_{p}
& & \widetilde{Y}_{C''} \ar[rd]^{\tilde f_{C''}} \ar@{-}[r] \ar@{<-->}[ll]
& \ar[r]& \widetilde{Y}_C \ar[rd]^{\tilde f_C} \ar@{-}[r]
& \ar[r]& \widetilde{\mathcal Y} \ar[rd]^{\tilde f} & \\
& & & C'' \ar[rr]^{\kappa}
& & C \ar[rr]
& & \widetilde{\mathcal X}
}
\]
By virtue of \cite[Def.~2.1(d)]{VZ03} 
and $Z'_{C'''}$ being the pull back of ${Z}'_C$ to some non-singular covering of $C$,  we have that $Z'_{C'''}$  is normal with rational Gorenstein singularities. 
The morphisms 
$pr_1 : S' \to C'''$ and $g'_{C'''}:{Z}'_{C'''}  \to C'''$ 
are  flat Gorenstein morphisms,  with $S'$ and $Z'_{C''''}$ birational, and $S'$ is also  normal and Gorenstein with at most rational singularities.

Using the same arguments as for \ref{L:VZ03Cor2.4ii}, one has an isomorphism  ${g'_{C'''}}_* \omega^{\otimes k}_{{Z}'_{C'''}/C'''} \cong {pr_{1}}_{*} \omega^{\otimes k}_{{S}'/C'''}$ (both are isomorphic to the push forward of $\omega_{\widetilde Z/C'''}^{\otimes k}$).  
 From this, one then has isomorphisms 
$${g'_{C'''}}_* \omega^{\otimes k}_{{Z}'_{C'''}/C'''} \cong {pr_{1}}_{*} \omega^{\otimes k}_{{S}'/C'''}\cong {pr_{1}}_{*} \tau_{{S}}^*\omega^{\otimes k}_{{S}/C''}\cong \tau_{C''}^* p_*\omega^{\otimes k}_{{S}/C''}, $$
where  for the second isomorphism we use the same argument as for \eqref{check-Kollar-condition} using \cite[2.68.2]{kollar_families}, and the last isomorphism uses that $\tau_{C''}$ is flat (being a surjective morphism of smooth curves).   Note that for the same reason we have $\kappa'^*g'_{C'*} \omega^{\otimes k}_{{Z}'_{C'}/C'} \cong g'_{C'''*} \omega^{\otimes k}_{{Z}'_{C'''}/C'''}$.
Consequently, applying $\kappa'^*$ to \eqref{incl2}, we obtain 
$$
\kappa'^*(\iota_C): \kappa'^*g'_{C'*} \omega^{\otimes k}_{{Z}'_{C'}/C'} = \tau_{C''}^* p_*\omega^{\otimes k}_{{S}/C''} \longrightarrow
\kappa'^*\tau_C^* (\tilde{f}_{C*} \omega^{\otimes k}_{\widetilde{{Y}}_C/C})= \tau_{C''}^*\kappa^*(\tilde{f}_{C*} \omega^{\otimes k}_{\widetilde{{Y}}_C/C});
$$
in particular, as both the source and target of $\kappa'^*(\iota_C)$ are obtained by pull back by $\tau_{C''}$, we have that   $\kappa'^*\det(\iota _C)=\det \kappa'^*(\iota_C)$ is pulled back from $C''$ via $\tau_{C''}$. 

Finally, since $\kappa$ is totally ramified at $x$, we can conclude from this that $\det (\iota_C)$ is pulled back from $C$ via $\tau_C$, completing the proof. 
\end{proof}

In the notation of \Cref{L:VZ03Cor2.4}, let $\mathcal Y^{(r)}$ denote a strong log resolution of singularities of the main component of the $r$-fold fibered product $\mathcal Y^r:=\mathcal Y\times_{\mathcal X}\cdots \times_{\mathcal X}\mathcal Y$, and let $f^{(r)}: \mathcal Y^{(r)}\to \mathcal X$ be the induced morphism. 
 From \cite[Lem. 2.2(ii)]{VZ03}, we have that the $r$-fold fibered product $g'^r: Z'^r :=  Z'\times_{ X'}\cdots \times_{ X'} Z' \to  X'$ is mild.  
 Define   
 \begin{equation}\label{E:tildf(r)-def}
 \tilde f^{(r)}: \widetilde {\mathcal Y}^{(r)}\to \widetilde {\mathcal X}
 \end{equation}
  to be a strong log resolution of singularities of the main component of $\widetilde {\mathcal X}\times _{\mathcal X} \mathcal Y^{(r)}$.
 For the normalization $ Y'^{(r)}$ of the main component of  the fibered product $ X'\times_{\widetilde {\mathcal X}}\widetilde {\mathcal Y}^{(r)}$, we choose strong resolution of singularities  $ Z^{(r)}$, and then we choose a strong log resolution of singularities of  $ Y''^{(r)}$,    resolving  the birational map between $ Z^{(r)}$ and $ Z'^{(r)}$, as well as the pull backs of the discriminants (i.e., the pull backs of the discriminants have normal crossings support).  

With this set-up, we provide the following generalization of \cite[Clm.~4.2]{VZ03}:

\begin{lem}

\label{L:VZ03Clm4.2}
In the notation of \Cref{L:VZ03Cor2.4} and the paragraph above (where we emphasize that $Z'^r$ is the self-fibered product over $X'$, not the product), there is an integer $r>0$ and a commutative diagram
\begin{equation}\label{E:VZS2diag-2}
\xymatrix{
 {Z'}^r  \ar[d]^{{g'}^r}&  {Y''}^{(r)}\ar[r]^{\rho^{(r)}} \ar[d]^{{f''}^{(r)}} \ar[l]_{\delta^{(r)}}&  Z^{(r)} \ar[r]^{\sigma^{(r)}} \ar[d]^{g^{(r)}}&  {Y'}^{(r)} \ar[r]^{{\tau'}^{(r)}} \ar[d]^{{f'}^{(r)}}& \widetilde {\mathcal Y}^{(r)}\ar[r]^{\widetilde \mu^{(r)}} \ar[d]^{\tilde f^{(r)}}&\mathcal Y^{(r)} \ar[d]^{f^{(r)}}\\
 X'\ar@{=}[r]&  X' \ar@{=}[r]&  X'\ar@{=}[r]&  X'\ar[r]^\tau& \widetilde {\mathcal X}\ar[r]^\mu& \mathcal X
}
\end{equation}
satisfying the conditions of \eqref{E:VZS2diag} in \Cref{P:VZ03S2} and such that: 
\begin{enumerate}[label=(\alph*)]
\item \label{L:VZ03Clm4.2a} For all $k\ge 1$ the sheaf $g_*^{(r)}\omega^{\otimes k}_{ Z^{(r)}/X'}$ is reflexive, and there is an isomorphism
$$
\xymatrix{
g_*^{(r)}\omega^{\otimes k}_{ Z^{(r)}/ X'} \ar[r]^<>(0.5){\sim}& \displaystyle \hat \bigotimes^rg_*\omega_{ Z/ X'}^{\otimes k}.
}
$$

\item \label{L:VZ03Clm4.2b} For all $k\ge 1$ there  is an inclusion
$$
\xymatrix{
g_*^{(r)}\omega^{\otimes k}_{ Z^{(r)}/ X'} \ar@{^(->}[r]& \tau^*\tilde f^{(r)}_*\omega_{\widetilde {\mathcal Y}^{(r)}/\widetilde {\mathcal X}}^{\otimes k}
}
$$
which is an isomorphism on $\tau^{-1}\mu^{-1}(\mathcal U)$ ($\mathcal U$ is the complement of the discriminant).

\item \label{L:VZ03Clm4.24.1.1} For all $k\ge 1$ there  is an isomorphism
$$
g'^r_*\omega_{ Z'^r/ X'}^{\otimes k}\cong \hat \bigotimes ^rg'_*\omega_{ Z'/ X'}^{\otimes k}.
$$  

\end{enumerate}
\end{lem}

\begin{proof}
The proof is essentially identical to \cite[Clm.~4.2]{VZ03}. 
The only assertion involving stacks, \ref{L:VZ03Clm4.2b}, and the first part of \ref{L:VZ03Clm4.2a}, follow from \Cref{L:VZ03Cor2.4}\ref{L:VZ03Cor2.4ii} and \ref{L:VZ03Cor2.4iii}.  For the second part of \ref{L:VZ03Clm4.2a}, we can use \Cref{L:VZ03Cor2.4}\ref{L:VZ03Cor2.4ii} to replace the left hand side by $g'^r_*\omega_{ Z'^r/ X'}^{\otimes k}$ and the right hand side by $\hat \bigotimes ^rg'_*\omega_{ Z'/ X'}^{\otimes k}$; now employ \ref{L:VZ03Clm4.24.1.1}.
\end{proof}

\begin{rem}\label{R:VZ-disc}
In the notation of \Cref{L:VZ03Cor2.4} and \Cref{L:VZ03Clm4.2},  
let $\mathbf \Delta\subseteq \mathcal X$ be a divisor containing the discriminant $\mathbf \Delta_f$ of $f$.  Let $\mathcal U=\mathcal X-\mathbf \Delta$ be the complement of $\mathbf \Delta$, so that in particular,  $f|_{\mathcal U}:\mathcal Y_{\mathcal U}\to \mathcal U$ is smooth.  Then $f^r|_{\mathcal U}:(\mathcal Y^r)|_{\mathcal U}\to \mathcal U$ is also smooth, so that by construction of $f^{(r)}: \mathcal Y^{(r)}\to \mathcal X$, we have $\mathbf \Delta_{f^{(r)}}\subseteq \mathbf \Delta$.  In addition, by the construction of $\widetilde {\mathcal Y}$ as a strong log resolution of the main component of the fiber product $\widetilde {\mathcal X}\times_{\mathcal X}\mathcal Y$, we have that  $\tilde f|_{\mu^{-1}(\mathcal U)}:\widetilde {\mathcal Y}|_{\mu^{-1}(\mathcal U)}\to \mu^{-1}(\mathcal U)$ is smooth.  Similarly,  $\tilde f^{(r)}|_{\mu^{-1}(\mathcal U)}:\widetilde {\mathcal Y}^{(r)}|_{\mu^{-1}(\mathcal U)}\to \mu^{-1}(\mathcal U)$ is smooth. It follows that $\mathbf \Delta_{\tilde f}\subseteq \mu^{-1}\mathbf \Delta$ and $\mathbf \Delta_{\tilde f^{(r)}} \subseteq \mu^{-1}\mathbf \Delta$.  In particular, if $\tilde {\mathbf \Delta}\subseteq \widetilde {\mathcal X}$ is any divisor containing the support of $\mu^{-1}\mathbf \Delta$, then $\mathbf \Delta_{\tilde f},\mathbf \Delta_{\tilde f^{(r)}}\subseteq \tilde {\mathbf \Delta}$. 

\end{rem}

\section{Some geometric constructions}\label{S:GeomConstRoots}

In order to make some geometric constructions with Hodge modules later, we will need some geometric constructions that modify our families to ensure the existence of nontrivial sections of certain line bundles over our families.  This generalizes some results from \cite{PS17, WW23}.

\subsection{A geometric construction for families with good minimal models} \label{S:s-geom-conPSVZ}
In this subsection we extend a geometric construction from \cite[Prop.~A.1]{PTW18} to the case of DM stacks (\Cref{P:PTW18-PA.1}).
The following result extends results of \cite[Thm.]{kollar_sub} and \cite[Thm.~1.1]{kawamata_kod_dim} to the situation of stacks; while technically we do not end up using this version, we include it for context.

\begin{lem}\label{L:preT:PS-4.2/21VZ}
 Let $f : \mathcal Y \to  \mathcal X$ be a surjective projective schematic  morphism of smooth proper integral  DM stacks over $\mathbb C$ with projective coarse moduli spaces, of maximal variation, with smooth and geometrically connected generic fiber.  
Assume that the  geometric generic fiber of $f$ is of general type, or, more generally,  the geometric generic fiber of $f$ admits a good minimal model.   
 Then  there is an integer  $m>0$ such that $\det (((f_*\omega_{  \mathcal Y/\mathcal  X})^{\otimes m})^{\vee\vee})$ is big. 
\end{lem}

\begin{proof}
Let $\tau:V\to \mathcal X$ be a finite flat morphism from a smooth projective variety $V$, let $Y_V$ be a strong log resolution of singularities of the main component of $\mathcal V\times_{\mathcal X}\mathcal Y$, and let $f_V:Y_V\to V$ be the induced morphism.  Then, from 
\Cref{L:LV3.2}, for all $k\ge 0$ there is a natural inclusion 
$$
 f_{V*}\omega_{\mathcal Y_V/V}^{\otimes k}\longrightarrow \tau^*f_*\omega_{\mathcal Y/\mathcal X} ^{\otimes k}, 
$$
which is an isomorphism over  the open subvariety $U'$ of $\mathcal X'$ where $\tau$ is \'etale.  Taking reflexive hulls gives an injection 
\begin{equation}\label{E:preT:PS-4.2/21VZ}
 (f_{V*}\omega_{\mathcal Y_V/V}^{\otimes k})^{\vee \vee}\longrightarrow (\tau^*f_*\omega_{\mathcal Y/\mathcal X} ^{\otimes k})^{\vee\vee},
\end{equation}
as the morphism is injective generically, and the sheaf on the left is torsion-free.  
We first note that, given our assumptions on the geometric generic fiber of $f$, which hold by base change for the geometric generic fiber of $f_V$, we have that  the determinant of $ (f_{V*}\omega_{\mathcal Y_V/V}^{\otimes k})^{\vee \vee}$ is big  (\cite[Thm.]{kollar_sub} and \cite[Thm.~1.1]{kawamata_kod_dim}).

Then, since the quotient of the morphism in \eqref{E:preT:PS-4.2/21VZ} is torsion, and is a quotient of a torsion-free sheaf, taking determinants (see \cite[Lem.~1.1]{CMZslope_stability}), we see that  the determinant of $(\tau^*f_*\omega_{\mathcal Y/\mathcal X} ^{\otimes k})^{\vee\vee}$ is big (it is big plus effective).  Using the functoriality of pull backs of determinants for torsion-free sheaves over finite flat morphisms (see \cite[\S 1.4]{CMZpositivity}), the above discussion implies that $\tau^*(\det (f_*\omega_{\mathcal Y/\mathcal X} ^{\otimes k})^{\vee\vee}))$ is big, which implies that $\det (f_*\omega_{\mathcal Y/\mathcal X} ^{\otimes k})^{\vee\vee})$ is big (see \cite[Lem.~2.5]{CMZpositivity}).  
\end{proof}

With this we can extend a geometric construction from \cite[Prop.~A.1]{PTW18} to the case of DM stacks:

\begin{pro}\label{P:PTW18-PA.1}
Let $f : \mathcal Y \to  \mathcal X$ be a surjective  schematic projective  morphism of smooth proper integral  DM stacks over $\mathbb C$ with projective coarse moduli spaces, of maximal variation, with smooth geometrically connected generic fiber, and with discriminant $\mathbf \Delta_f$ contained in a  normal crossings divisor $\mathbf \Delta\subseteq \mathcal X$.  
Assume that the  geometric generic fiber of $f$ is of general type, or, more generally,   admits a good minimal model.  Further, assume that $\X$ has generically trivial stabilizers.

Then there is a blow-up $\mu:\widetilde {\mathcal X}\to \mathcal X$, integers $r,m>0$, 
 and a surjective schematic  morphism (see \eqref{E:tildf(r)-def})
 $$
\tilde f^{(r)} : \widetilde {\mathcal Y}^{(r)} \to   \widetilde {\mathcal X}
$$
 of  smooth proper integral DM stacks over $\mathbb C$ with projective coarse moduli spaces, 
such that the pre-image $\mu^{-1}\mathbf \Delta$ is an snc divisor with support containing the discriminant $\mathbf \Delta_{\tilde f^{(r)}}$, and for any snc divisor $\tilde {\mathbf \Delta}$ containing $\mu^{-1}\mathbf \Delta$, there exists 
 a line bundle  
 $\tilde{\mathcal A}$ on $\widetilde {\mathcal X}$ such that 
the line bundle   $\tilde{\mathcal A}(-\tilde {\mathbf \Delta})$ is big, and 
setting 
\begin{equation}\label{E:P:PTW18-PA.1}
  \widetilde{\mathcal B}:=\omega_{ \widetilde {\mathcal Y}^{(r)}/ \widetilde {\mathcal X}}\otimes  \tilde f^{(r)*}\tilde{\mathcal A}^{\otimes -1},
\end{equation}
the sheaf $\tilde f^{(r)}_*(\widetilde {\mathcal B}^{\otimes m})$ is globally generated over the open substack $ \widetilde { U}^\circ$ of  $\widetilde {\mathcal U}:=\widetilde {\mathcal X}-\tilde {\mathbf \Delta}$ having trivial stabilizers ($\widetilde U^\circ$ is isomorphic to a quasi-projective variety).

Moreover, if the fibers of $f$ over the complement of $\mathbf \Delta$ have semi-ample canonical bundle, then $\widetilde {\mathcal B}^{\otimes m}$ is also generated by global sections over $(\tilde f^{(r)})^{-1}(\widetilde { U}^\circ)$. 
\end{pro}

\begin{proof}
The proof is essentially identical to \cite[Prop.~A.1]{PTW18}, which builds off of the strategy in \cite[Prop.~4.1 and Cor.~4.3]{VZ03}.   The proof is somewhat lengthy, and requires some modifications in the situation of stacks, and so we break the proof into several steps for clarity.

\subsubsection*{Step 1: Geometric set-up via mild reduction, part I} 
To start we fix a mild reduction diagram for $f:\mathcal Y\to \mathcal X$ as in \eqref{E:VZ03Cor2.4} of \Cref{L:VZ03Cor2.4}:
\begin{equation*}
\xymatrix{
Z'  \ar[d]^{g'}&  Y''\ar[r]^\rho \ar[d]^{f''} \ar[l]_{\delta}&  Z \ar[r]^\sigma \ar[d]^g&  Y' \ar[r]^{\tau'} \ar[d]^{f'}& \widetilde {\mathcal Y}\ar[r]^{\widetilde \mu} \ar[d]^{\tilde f}&\mathcal Y \ar[d]^f\\
 X'\ar@{=}[r]&  X' \ar@{=}[r]&  X'\ar@{=}[r]&  X'\ar[r]^\tau& \widetilde {\mathcal X}\ar[r]^\mu& \mathcal X.
}
\end{equation*} 
Let $\tilde \pi:\widetilde {\mathcal X}\to \widetilde X$ be the coarse moduli space, and 
choose an ample line bundle $\tilde L$ on $\widetilde X$ so that $\tilde {\mathcal L}:=\tilde \pi^*L$ satisfies the conditions that     $\tau^*\tilde {\mathcal L}\otimes \omega_{X'}^{-1}$ 
separates $2n$-jets on $X'$, where $n=\dim \mathcal X$, and  
$\tau_*\mathcal O_{X'}\otimes \tilde {\mathcal L}$ is generated by global sections over a Zariski open substack of $\mathcal X$.     
 For the condition on $2n$-jets, recall that $\tilde \pi \circ \tau:X'\to \widetilde X$ is a finite morphism of projective varieties.   The global generation assertion is where we are using that the stabilizers of $\mathcal X$, and therefore $\widetilde {\mathcal X}$, are generically trivial.  More precisely, we can choose $\tilde L$ such that $\tilde L\otimes \widetilde \pi_*\tau_*\mathcal O_{ X'}$ is globally generated on $\widetilde X$, so that there is a surjection
$$
\bigoplus \mathcal O_{\widetilde X}\twoheadrightarrow \tilde L \otimes \widetilde \pi_*\tau_*\mathcal O_{ X'}.
$$
Applying $\widetilde \pi^*$, which is right exact, and post composing with the co-unit of adjunction, we have
$$
\bigoplus \mathcal O_{\widetilde {\mathcal X}}\twoheadrightarrow \tilde {\mathcal L}\otimes \widetilde \pi^*\widetilde \pi_*\tau_*\mathcal O_{ X'}\to \tilde {\mathcal L}\otimes \tau_*\mathcal O_{ X'} 
$$ 
where we use that the co-unit of adjunction is generically an isomorphism, since $\widetilde \pi$ is generically an isomorphism (the stabilizers of $\widetilde {\mathcal X}$ are generically trivial). 

Note that from \cite[Thm.]{kollar_sub} and \cite[Thm.~1.1]{kawamata_kod_dim} (e.g., \Cref{L:preT:PS-4.2/21VZ})
we have that the line bundle  $ \operatorname{det} ((g_*
\omega^{\otimes v}_{ Z/ X'})^{\vee\vee})$ is big, and so, together with \Cref{L:VZ03Cor2.4} \ref{L:VZ03Cor2.4iv}, which implies that some sufficiently large tensor power is the pull back via $\tau$ of a line bundle on $\widetilde {\mathcal X}$,   one can conclude that   for some sufficiently large natural number $N_v$, one has
\begin{equation}\label{E:detgtau*}
 (\operatorname{det} (g_*
\omega^{\otimes v}_{ Z/ X'}))^{\otimes N_{k}}\cong \tau^*(\tilde{\mathcal L}(\tilde {\mathbf \Delta} +\mathcal D))
\end{equation}
for some effective divisor $\mathcal D$ on $\widetilde {\mathcal X}$.

\subsubsection*{Step 2: Geometric set-up via mild reduction, part II} 
Let $\mathcal U:= \X-\mathbf \Delta$, let $\widetilde {\mathcal U}:= \widetilde {\mathcal X}-\tilde {\mathbf \Delta}$, and let $U':=\tau^{-1}\widetilde {\mathcal U}\subseteq X'$, and for each point $x'$ of $U'$, let $Z_{x'}$ be the fiber of $g$ over $x'$, which by assumption is a smooth projective variety.   Define a function $e$ on $U'$ by 
\[e(x'):=e(\omega^{\otimes v}_{Z_{x'}}):=\sup\left\{\frac{1}{\textup{lct}({E})} :  {E}\in \left|\omega^{\otimes v}_{Z_{x'}}\right|\right\}.
\]
The function $e$ is upper semi-continuous, and  there exists a positive integer $C$ such that for all $x'$ in $U'$ we have $$e(\omega^{\otimes v}_{Z_{x'}})< Cv.$$   
Recall that the upper semi-continuity follows from the lower semi-continuity of the log canonical threshold of relative divisors for smooth proper morphisms \cite[Prop.~5.17]{Vmoduli}, combined with 
the invariance of plurigenera, as using the invariance of plurigenera, one can show that any given divisor in question  on a fiber lifts to a relative divisor locally over the base of the family.  The uniform bound follows from  \cite[Cor.~5.11]{Vmoduli}. 

Now take $r=C(C+1)vN_vr_0$, where, in the notation of  \eqref{E:VZ03Cor2.4},  $r_0=\operatorname{rank}(g_{*} \omega^{\otimes v} _{Z'/X'} )$, 
and consider the mild reduction diagram \eqref{E:VZS2diag-2} for the $r$-fold fibered product of $f$, as in \Cref{L:VZ03Clm4.2}:
\begin{equation}\label{E:VZS2diag-2-proof(r)}
\xymatrix{
 {Z'}^r  \ar[d]^{{g'}^r}&  {Y''}^{(r)}\ar[r]^{\rho^{(r)}} \ar[d]^{{f''}^{(r)}} \ar[l]_{\delta^{(r)}}&  Z^{(r)} \ar[r]^{\sigma^{(r)}} \ar[d]^{g^{(r)}}&  {Y'}^{(r)} \ar[r]^{{\tau'}^{(r)}} \ar[d]^{{f'}^{(r)}}& \widetilde {\mathcal Y}^{(r)}\ar[r]^{\widetilde \mu^{(r)}} \ar[d]^{\tilde f^{(r)}}&\mathcal Y^{(r)} \ar[d]^{f^{(r)}}\\
 X'\ar@{=}[r]&  X' \ar@{=}[r]&  X'\ar@{=}[r]&  X'\ar[r]^\tau& \widetilde {\mathcal X}\ar[r]^\mu& \mathcal X.
}
\end{equation}

\subsubsection*{Step 3: Global generation of the push forward of the relative pluricanonical sheaf for $g^{(r)}$ (see \eqref{E:GGg*wZr/X'})}

The next steps of the proof all focus on the left hand portion of the diagram above, where all the spaces are projective varieties, and the arguments are therefore identical to those of \cite[Prop.~A.1]{PTW18}. We will need some of the construction, and so we recall this briefly here.  Using \cite[Cor.~2.4(vii)]{VZ03} (\Cref{L:VZ03Cor2.4}\ref{L:VZ03Cor2.4ii}), one has $g'_*\omega_{ Z'/ X'}^{\otimes v} \cong g_*\omega_{ Z/ X'}^{\otimes v}$ and $g'^r_*\omega_{ Z'^r/ X'}^{\otimes v} \cong g^{(r)}_*\omega_{ Z^{(r)}/ X'}^{\otimes v}$, and by \cite[Lem.~3.5]{Vkod1} (\Cref{L:LV3.5}), one has the isomorphism $g'^r_*\omega^{\otimes v}_{Z'^r/X'}\cong (\bigotimes ^rg'_*\omega_{ Z'/ X'}^{\otimes v})^{\vee\vee}$.  

In addition, via the standard arguments regarding determinants and tensor products  of vector bundles, together with push forwards of vector bundles over codimension $2$ loci, there is a natural inclusion  $(\det g_*\omega_{ Z/ X'}^{\otimes v}) \hookrightarrow (\bigotimes ^{r_0}g'_*\omega_{ Z'/ X'}^{\otimes v})^{\vee\vee}$  that is split over $U'$.  Putting this all together with \eqref{E:detgtau*} one obtains an inclusion 
$$
 \tau^*(\tilde{\mathcal L}(\tilde{\mathbf \Delta} +\mathcal D))^{\otimes C(C+1)v}
 \cong  (\operatorname{det} (g_*
\omega^{\otimes v}_{ Z/ X'}))^{\otimes C(C+1)vN_{v}}\hookrightarrow g^{(r)}_*\omega_{ Z^{(r)}/ X'}^{\otimes v}
 $$
 that is also split over $U'$. This corresponds to an effective divisor $$\Gamma\in \left|\omega ^{\otimes v}_{Z^{(r)}/X'}\otimes g^{(r)*}\tau^*\tilde{\mathcal L}(\tilde{\mathbf \Delta} +\mathcal D)^{\otimes -C(C+1)v}\right|$$ that does not contain the fiber $Z^{(r)}_{x'}$ for every $x'\in U'$.  Since $Z^{(r)}_{x'}=(Z_{x'})^r=Z_{x'}\times \cdots \times Z_{x'}$, using the bound $e(\omega^{\otimes v}_{Z_{x'}})< Cv$, and  \cite[Cor.~5.21]{Vmoduli}, we have $\operatorname{lct}(\Gamma|_{Z^{(r)}_{x'}})>1/Cv$ for every $x'$ in $U'$.   
 
 We can now apply \cite[Thm.~A.2]{PTW18}. In the notation of  that theorem, take  $$\mathscr M=\omega ^{\otimes kv}_{Z^{(r)}/X'}\otimes g^{(r)*}\tau^*\tilde{\mathcal L}(\tilde{\mathbf \Delta} +\mathcal D)^{\otimes -kC(C+1)v}$$ with the natural singular metric induced by the effective divisor $k\Gamma$, take $\mathscr B= \tau^*\tilde{\mathcal L}$, and take the index $k$ in the theorem to be $kCv$.   Using that, up to the pull back of a line bundle on the base, $\mathscr M$ agrees with $\omega ^{\otimes kv}_{Z^{(r)}/X'}$, one has that the fibers of  $\mathscr M$ over points in $U'$  are pluricanonical bundles, so that the invariance of plurigenera,  together with the fact that $\left(\frac{1}{Cv}\Gamma\right)|_{Z^{(r)}_{x'}}$ is klt for all $x'$ in $U'$,  imply  that the hypotheses of the theorem are satisfied.   Applying \cite[Thm.~A.2]{PTW18} gives that the sheaf  
\begin{equation}\label{E:GGg*wZr/X'}
g^{(r)}_*\omega ^{\otimes k(C+1)v}_{Z^{(r)}/X'}\otimes \tau^*\tilde {\mathcal L}(\tilde{\mathbf \Delta} +\mathcal D)^{\otimes -kC(C+1)v}\otimes \tau^*\tilde {\mathcal L}
\end{equation}
is generated by global sections over $U'$.

\subsubsection*{Step 4: Global generation of the push forward of the relative pluricanonical sheaf for $\tilde f^{(r)}$ (see \eqref{E:L:VZ03proof319})}
Now, returning to the case of stacks, using \Cref{L:VZ03Clm4.2}\ref{L:VZ03Clm4.2b} there is an inclusion  
\begin{equation}\label{E:L:VZ03Clm4.2b-proof}
g^{(r)}_*\omega ^{\otimes k(C+1)v}_{Z^{(r)}/X'}\hookrightarrow \tau^*\tilde f^{(r)}_*\omega^{\otimes k(C+1)v}_{\widetilde {\mathcal Y}^{(r)}/\widetilde {\mathcal X}},
\end{equation}
which is an isomorphism over $U'$.  
Combining with the global generation of \eqref{E:GGg*wZr/X'}, there is a morphism  
\begin{equation}\label{E:L:VZ03proof316}
\bigoplus \mathcal O_{X'}\longrightarrow  \tau^*\tilde f^{(r)}_*\omega^{\otimes k(C+1)v}_{\widetilde {\mathcal Y}^{(r)}/\widetilde {\mathcal X}}\otimes \tau^*\tilde {\mathcal L}(\tilde{\mathbf \Delta} +\mathcal D)^{\otimes -kC(C+1)v}\otimes \tau^*\tilde {\mathcal L}
\end{equation}
that is a surjection over $U'$.   Applying $\tau_*$, and the projection formula on the right, we obtain a morphism 
\begin{equation}
\bigoplus \tau_*\mathcal O_{X'}\longrightarrow  \tilde f^{(r)}_*\omega^{\otimes k(C+1)v}_{\widetilde {\mathcal Y}^{(r)}/\widetilde {\mathcal X}}\otimes \tilde {\mathcal L}(\tilde{\mathbf \Delta} +\mathcal D)^{\otimes -kC(C+1)v}\otimes \tilde {\mathcal L}\otimes \tau_*\mathcal O_{X'},
\end{equation}
which is still a surjection over $\widetilde {\mathcal U}$.   Tensoring by $\tilde {\mathcal L}^{k(C+1)v-1}$, and recalling that by assumption $\tau_*\mathcal O_{X'}\otimes \tilde {\mathcal L}^{k(C+1)v-1}$ is generated by global sections over the open substack $\widetilde {\mathcal U}^\circ$  of $\widetilde {\mathcal U}$ having trivial stabilizers, we have a morphism 
\begin{equation}
\bigoplus \mathcal O_{\widetilde {\mathcal X}} \longrightarrow  \tilde f^{(r)}_*\omega^{\otimes k(C+1)v}_{\widetilde {\mathcal Y}^{(r)}/\widetilde {\mathcal X}}\otimes \tilde {\mathcal L}^{\otimes -k(C-1)(C+1)v}\otimes \mathcal O_{\widetilde {\mathcal X}}(-kC(C+1)v(\tilde{\mathbf \Delta} +\mathcal D))\otimes  \tau_*\mathcal O_{X'},
\end{equation}
which is a surjection over $\widetilde {\mathcal U}^\circ$.  Finally, via the trace map $\tau_*\mathcal O_{X'}\to \mathcal O_{\widetilde {\mathcal X}}$, we obtain a morphism  
\begin{equation}\label{E:L:VZ03proof319}
\bigoplus \mathcal O_{\widetilde {\mathcal X}} \longrightarrow  \tilde f^{(r)}_*\omega^{\otimes k(C+1)v}_{\widetilde {\mathcal Y}^{(r)}/\widetilde {\mathcal X}}\otimes \tilde {\mathcal L}^{\otimes -k(C-1)(C+1)v}\otimes \mathcal O_{\widetilde {\mathcal X}}(-kC(C+1)v(\tilde{\mathbf \Delta} +\mathcal D)),
\end{equation}
which is a surjection over $\widetilde {\mathcal U}^\circ$. 
 This completes the proof of the first claim of the proposition.  Indeed, we take  $\tilde {\mathcal A}:=\tilde {\mathcal L}^{\otimes(C-1)}\otimes \mathcal O_{\widetilde {\mathcal X}}(C(\tilde{\mathbf \Delta} +\mathcal D))$, and   $m:=k(C+1)v$.

\subsubsection*{Step 5: Global generation of the relative pluricanonical sheaf for $\tilde f^{(r)}$}
The second claim follows from the first, recalling that $\widetilde { U}^\circ$ is a variety,  so that $\tilde f^{(r)}|_{(\tilde f^{(r)})^{-1} \widetilde { U}^\circ}:(\tilde f^{(r)})^{-1} \widetilde { U}^\circ \to \widetilde { U}^\circ $ is a morphism of varieties,  and then noting that the assumption on the fibers  implies, after possibly replacing $k$ with a sufficiently large multiple to get a uniform global generation statement,  that the
natural map $$\tilde f^{(r)*}\tilde f^{(r)}_*\omega^{\otimes m}_{\widetilde {\mathcal Y}^{(r)}/\widetilde {\mathcal X}} \to \omega^{\otimes m}_{\widetilde {\mathcal Y}^{(r)}/\widetilde {\mathcal X}}$$ 
is surjective over $(\tilde f^{(r)})^{-1}\widetilde { U}^\circ$. 
\end{proof}

We now make the following observation following \cite[Rem.~A.3]{PTW18}:

\begin{rem}\label{R:PTW18-RA.3}
Recalling that $g'^r_*\omega_{ Z'^r/ X'}^{\otimes m} \cong g^{(r)}_*\omega_{ Z^{(r)}/ X'}^{\otimes m}$, then 
applying $\tau_{*}$ to \eqref{E:L:VZ03Clm4.2b-proof}, and finally utilizing the  
trace map for $\tau$,  we have a morphism 
\[ \tau_{*} g'^r_{*}\omega^{\otimes m}_{ Z'^r/{X}'} \longrightarrow \tau_{*}\tau^*\tilde f_{*}^{(r)}\omega^{\otimes m}_{\widetilde{\Y}^{(r)}/\widetilde{\X}_1} \longrightarrow \tilde f_{*}^{(r)}\omega^{\otimes m}_{\widetilde{\Y}^{(r)}/\widetilde{\X}}.
\]
Tensoring by $-\otimes \tilde {\mathcal A}^{\otimes -m}$, we conclude the existence of a morphism
\[\tau_{*} g'^r_{*}\omega^{\otimes m}_{ Z'^r/{X}'}\otimes \tilde {\mathcal A}^{\otimes-m}\longrightarrow  \tilde f_{*}^{(r)}\omega^{\otimes m}_{\widetilde{\Y}^{(r)}/\widetilde{\X}}\otimes \tilde{\mathcal A}^{\otimes-m}
\]
This induces a natural map on global sections, and we denote by $\mathbb V_m$ the image of this map:
\begin{equation}\label{E:VVmdefinition}
\xymatrix{
H^0(X',g'^r_*\omega_{ Z'^r/ X'}^{\otimes m} \otimes \tau^*\tilde{\mathcal A}^{\otimes -m}) \ar[r] \ar@{->>}[rd]& H^0(\widetilde {\mathcal X},\tilde f^{(r)}_*\omega^{\otimes m}_{\widetilde {\mathcal Y}^{(r)}/\widetilde {\mathcal X}}\otimes \tilde{\mathcal A}^{\otimes -m})\\
& \mathbb V_m \ar@{^(->}[u].\\
}
\end{equation}
The proof of the proposition above, and in particular the computations in \eqref{E:L:VZ03proof316}--\eqref{E:L:VZ03proof319}, show more precisely that $ \tilde f^{(r)}_*\omega^{\otimes m}_{\widetilde {\mathcal Y}^{(r)}/\widetilde {\mathcal X}}\otimes \tilde{\mathcal A}^{\otimes -m}$ is generated over $\widetilde { U}^\circ$ by the sections belonging to the subspace $\mathbb V_m$.

\end{rem}

\subsubsection{A further log resolution of the base}
For the arguments later with Hodge modules and Higgs bundles, we will actually need to make a further modification of the base of the family. 
For this we generalize \cite[Prop.~A.4]{PTW18}: 

\begin{pro}\label{P:PTW18-PA.4}
With notation as in the statement and proof of \Cref{P:PTW18-PA.1} and \Cref{R:PTW18-RA.3}, let $\widetilde {\mathcal S}\subseteq \widetilde {\mathcal X}$ be an effective divisor  containing the discriminant  $\mathbf \Delta_\tau$ of $\tau$ (see diagram \eqref{E:VZS2diag-2-proof(r)}).  Let $ \mu_1:\widetilde {\mathcal X}_1\to \widetilde{\mathcal{X}}$ be a log resolution of the pair $(\widetilde{\mathcal{X}},\tilde {\mathbf \Delta}+\widetilde {\mathcal S})$ with centers contained in $\mathcal S$  such that 
$\mu_1^{-1}\mathbf \Delta_\tau$ is also normal crossings.

  Then there is a commutative diagram
$$
\xymatrix{
\widetilde{\mathcal Y}^{(r)}_1 \ar[r]^<>(0.5){\tilde \mu_1} \ar[d]^{{\tilde f}_1^{(r)}}& \widetilde{\mathcal{Y}}^{(r)} \ar[d]^{{\tilde f}^{(r)}}\\
\widetilde{\mathcal X}_1\ar[r]^{\mu_1}&\widetilde{\mathcal{X}}
}
$$
and a closed substack $\mathbf T\subseteq \widetilde {\mathcal X}_1$ of codimension at least $2$ that satisfies the following conditions, where we fix the notation  $\widetilde {\mathcal X}_1^\circ:=\widetilde {\mathcal X}_1-\mathbf T$, $\widetilde {\mathcal Y}_1^{(r)\circ}:=\widetilde {\mathcal Y}_1^{(r)}-\tilde f^{(r)-1}(\mathbf T)$,  and $\tilde f^{(r)\circ} =\tilde f^{(r)}|_{\widetilde {\mathcal Y}^{(r)\circ}}$:

\begin{enumerate}[label=(\alph*)]

\item  The morphism  $\tilde f_1^{(r)}: \widetilde {\mathcal Y}^{(r)}_1\to \widetilde {\mathcal X}_1$ is  schematic and projective from a smooth proper integral DM stack $\widetilde {\mathcal Y}^{(r)}_1$ over $\mathbb C$ with projective coarse moduli space.

\item  \label{P:PTW18-PA.4-2}

The morphism 
$\tilde \mu_1 $ is schematic projective and birational with divisorial exceptional locus, and setting $\widetilde {\mathcal E}_1\subseteq \widetilde {\mathcal X}_1$ to be the exceptional locus of $\mu_1$ (which by construction is contained in $\mu_1^{-1}(\mathcal S$)), $\widetilde {\mathbf Z}:= \mu_1(\widetilde {\mathcal E}_1)\subseteq \widetilde {\mathcal X}$,  $\widetilde {\mathcal U}_1:=\widetilde {\mathcal X}_1-\widetilde {\mathcal E}_1$, $\widetilde {\mathcal U}:=\widetilde {\mathcal X}-\widetilde {\mathbf Z}$,  $\widetilde {\mathcal Y}^{(r)}_{1,\widetilde {\mathcal U}_1}:=(\tilde f_1^{(r)})^{-1}(\widetilde {\mathcal U}_1)$, and $\widetilde {\mathcal Y}^{(r)}_{\widetilde {\mathcal U}}:=(\tilde f^{(r)})^{-1}(\widetilde {\mathcal U})$, we have that $\tilde \mu_1  $ restricts to an isomorphism 
 $
 \tilde \mu_1 |_{\widetilde {\mathcal Y}^{(r)}_{1,\widetilde {\mathcal U}_1}}:\widetilde {\mathcal Y}^{(r)}_{1,\widetilde {\mathcal U}_1} \stackrel{\sim}{\rightarrow} \widetilde {\mathcal Y}^{(r)}_{\widetilde {\mathcal U}} $
 over  $\mu_1|_{\widetilde {\mathcal U}_1}: \widetilde {\mathcal U}_1\stackrel{\sim}{\to} \widetilde {\mathcal U}$.

\item Recall that $\widetilde  {\mathcal B}= \omega_{\widetilde {\mathcal Y}^{(r)}/\widetilde {\mathcal X}}\otimes  \tilde f^{(r)*}\tilde {\mathcal A}^{-1}$ and define $\widetilde {\mathcal B}_1:= \omega_{\widetilde {\mathcal Y}_1^{(r)}/\widetilde {\mathcal X}_1}\otimes  \tilde f^{(r)*}_1\tilde {\mathcal A}_1^{-1}$, where $\tilde {\mathcal A}_1:= \mu_1^*\tilde {\mathcal A}$. 
 Under the identifications  of \ref{P:PTW18-PA.4-2}, we have   that  
 $  \widetilde  {\mathcal B}_1|_{\widetilde {\mathcal Y}_{\widetilde {\mathcal U}_1}}=\widetilde {\mathcal B}|_{\widetilde {\mathcal Y}_{\mathcal U}}$, and for any section $s\in \mathbb{V}_m\subseteq  H^0(\widetilde{\mathcal{X}}, \tilde{f}_{*}^{(r)}\widetilde{\mathcal B}^{\otimes m})$ (see \Cref{R:PTW18-RA.3} and \eqref{E:VVmdefinition}), there exists a
section $\tilde s\in H^0(\widetilde{\mathcal{X}}_1^\circ, \tilde{f}_{1*}^{(r)}\widetilde{\mathcal B}_1^{\otimes m})$ such that

\begin{equation}\label{E:P:PTW18-PA.4-3}
 \tilde s|_{\widetilde {\mathcal X}_1^\circ-\mu_1^{-1}\widetilde {\mathcal S}}=\tilde \mu_1^*s|_{\widetilde {\mathcal X}_1^\circ-\mu_1^{-1}\widetilde {\mathcal S}}.
 \end{equation}

\end{enumerate}
Moreover,  letting $\tilde {\mathbf \Delta}_1$ be the reduced snc divisor with support  $\mu_1^{-1}\tilde {\mathbf \Delta}$,  we have that $\tilde {\mathcal A}_1(-\tilde {\mathbf \Delta}_1)$ is big.

\end{pro}

\begin{proof}

Let $\mu_1: \widetilde{\mathcal X}_1\to \widetilde{\mathcal{X}}$
be the strong log resolution of the pair $(\widetilde {\mathcal X},\tilde {\mathbf \Delta}+\mathcal S)$ as in the statement of \Cref{P:PTW18-PA.4}.  
Consider the diagram \eqref{E:VZS2diag-2-proof(r)}, and pull back all but the rightmost morphisms in that diagram via $\mu_1$, modifying the diagram as follows: Set $\widetilde{ {X}}_1''$ to be the normalization of the main component of $X'\times_{\widetilde{\mathcal{X}}} \widetilde{ \mathcal{X}}_1$ and $\widetilde\tau'\colon \widetilde{{X}}''_1\to \widetilde{\mathcal{X}}_1$ the induced finite map. We precompose the map $\widetilde{{X}}_1''\rightarrow X'$ with a desingularization $\mu' \colon  X'_1\to \widetilde X''_1$, and get a birational map $\widetilde\mu'\colon X'_1\to X'$.  We then take $\widetilde \Y_1^{(r)}$ to be a desingularization of the main component of 
$\widetilde \Y^{(r)}\times_{\widetilde\X}\widetilde \X_1$, so that the induced morphism $\tilde f^{(r)}\colon \widetilde \Y_1^{(r)}\to \widetilde \X_1$ is a birational model of $f^{(r)}$. 

We obtain a mild reduction type diagram for the new family 
$\tilde f^{(r)}_1$:
\begin{equation}\label{E:MRZ1'r}
\begin{tikzcd}
Z'^r_1 \arrow[d, "g'^{r}_1"] & Y''^{(r)}_1 \arrow[r, "\rho^{(r)}_1"] \arrow[d, "f''^{(r)}_1"] \arrow[l, "\delta^{(r)}_1"'] & Z^{(r)}_1 \arrow[r, "\sigma^{(r)}_1"] \arrow[d, "g^{(r)}_1"] & Y'^{(r)}_1 \arrow[r, "\tau'^{(r)}_1"] \arrow[d, "f'^{(r)}_1"] & \widetilde{\Y}^{(r)}_1 \arrow[d, "\tilde f^{(r)}_1"] \\
X'_1 \arrow[r, "="]          & X'_1 \arrow[r, "="]                                                                         & X'_1 \arrow[r, "="]                                          & X'_1 \arrow[r, "\tau_1"]                                      & \widetilde \X_1                                         
\end{tikzcd}
\end{equation}
where $\tau_1=\widetilde\tau'\circ \mu'$, is a generically finite morphism, $Z'_1=Z'\times_{X'} X_1'$ and $ g'^r_1$ is the induced mild morphism, 
and $ \sigma_1^{(r)}$, $\rho_1^{(r)}$ and $\delta_1^{(r)}$ are similar to those in the mild reduction diagram for $\tilde f^{(r)}$. 
One important caveat here, why we are \emph{not} calling \eqref{E:MRZ1'r} a mild reduction diagram,  is that $\tau_1$ is only \emph{generically} finite.  
 From the construction,  we can conclude  that there  exists a closed substack $\mathbf T\subseteq \widetilde {\mathcal X}_1$ of codimension at least $2$ 
such that $\tau_1$ is finite and flat over ${\widetilde {\mathcal X}_1 - \mathbf T}$. 
For later reference, recalling that  $\widetilde {\mathcal X}^\circ_1=\widetilde {\mathcal X}_1-\mathbf T$, let $X_1'^\circ:=\tau_1^{-1}\widetilde {\mathcal X}_1^\circ\subseteq X_1'$.

The first consequence we draw from all of this is that we have a Cartesian diagram 
\[\begin{tikzcd}
Z'^r_1   \arrow[r, "\tilde\nu'"] \arrow[d, "g'^r_1"'] & Z'^r \arrow[d, "g'^r"'] \\
X_1' \arrow[r, "\tilde\mu'"]                         & X',                     
\end{tikzcd}
\]
and so, by the mildness of the vertical morphisms in the diagram above, and  arguing as in the proof of \eqref{check-Kollar-condition}, we have an isomorphism 

\begin{equation}\label{Comparison-pull back-of-mild'}
\widetilde\nu'^*\omega^{\otimes m}_{Z'^r/X'} \cong \omega^{\otimes m}_{ Z'^r_1/ X_1'}.
\end{equation}

The second consequence is that, from \Cref{L:VZ03Clm4.2}\ref{L:VZ03Clm4.2b},  there is an inclusion  
\begin{equation}\label{E:L:PTWClmA.3-proof'}
g^{(r)}_{1*}\omega ^{\otimes m}_{Z_1'^{(r)}/X_1'}|_{X_1'^\circ}\hookrightarrow \tau_1^*\tilde f^{(r)}_{1*}\omega^{\otimes m}_{\widetilde {\mathcal Y}^{(r)}_1/\widetilde {\mathcal X}_1}|_{X_1'^\circ}
\end{equation}
that agrees with \eqref{E:L:VZ03Clm4.2b-proof} over $\tau^{-1}(\widetilde {\mathcal X}-\mathcal S)$, as $\mu_1$ is the identity over $\widetilde {\mathcal X}-\mathcal S$.

Now, by the definition of the space  $\mathbb{V}_m$, the section $s\in H^0(\widetilde {\mathcal X},\tilde f^{(r)}_*\omega^{\otimes m}_{\widetilde {\mathcal Y}^{(r)}/\widetilde {\mathcal X}}\otimes \tilde{\mathcal A}^{\otimes -m})$ lifts to some section  $s'\in H^0(X',g'^r_*\omega_{ Z'^r/ X'}^{\otimes m} \otimes \tau^*\tilde{\mathcal A}^{\otimes -m}) $, and by  \eqref{Comparison-pull back-of-mild'}, this lifts to  a section 
\begin{align*}
\tilde{s}':=\tilde{\mu}'^*s'&\in H^0\big(X'_1,g'^r_{1*}\omega^{\otimes m}_{Z'^r/X'_1}\otimes\tilde\mu'^* \tau^*\tilde{\mathcal A}^{\otimes -m}\big)\\
&=H^0\big(X'_1,g'^r_{1*}\omega^{\otimes m}_{Z'^r/X'_1}\otimes \tau_1^*\mu_1^*\tilde{\mathcal A}^{\otimes -m}\big)\\
&= H^0(\widetilde {\mathcal X}_1, \tau_{1*} g'^r_{1*}\omega^{\otimes m}_{ Z_1'^r/{X}_1'}\otimes \mu_1^*\tilde{\mathcal A}^{\otimes-m}).
\end{align*}
At the same time, applying $\tau_{1*}$ to \eqref{E:L:PTWClmA.3-proof'}, and then utilizing the  
trace map for $\tau_1$,  we have a morphism 
\[ \tau_{1*} g'^r_{1*}\omega^{\otimes m}_{ Z_1'^r/{X}_1'}|_{\widetilde {\mathcal X}_1^\circ}\longrightarrow \tau_{1*}\tau_1^*\tilde f_{1*}^{(r)}\omega^{\otimes m}_{\widetilde{\Y}^{(r)}_1/\widetilde{\X}_1}|_{\widetilde {\mathcal X}_1^\circ}\longrightarrow \tilde f_{1*}^{(r)}\omega^{\otimes m}_{\widetilde{\Y}^{(r)}_1/\widetilde{\X}_1}|_{\widetilde {\mathcal X}_1^\circ}.
\]
Finally, tensoring by $-\otimes \mu_1^*\tilde{\mathcal A}^{\otimes -m}$, we conclude the existence of a morphism
\[\eta: \tau_{1*} g'^r_{1*}\omega^{\otimes m}_{ Z_1'^r/{X}_1'}\otimes \mu_1^*\tilde{\mathcal A}^{\otimes-m}|_{\widetilde {\mathcal X}_1^\circ}\longrightarrow  \tilde f_{1*}^{(r)}\omega^{\otimes m}_{\widetilde{\Y}^{(r)}_1/\widetilde{\X}_1}\otimes \mu_1^*\tilde{\mathcal A}^{\otimes-m}|_{\widetilde {\mathcal X}_1^\circ}
\]
and define 
\[
\tilde s:=\eta(\tilde s'|_{\widetilde {\mathcal X}_1^\circ})\in H^0\big(\widetilde{\mathcal X}^\circ_1,\tilde f_{1*}^{(r)}\omega^{\otimes m}_{\widetilde{\Y}^{(r)}_1/\widetilde{\X}_1}\otimes \mu_1^*\tilde{\mathcal A}^{\otimes-m}|_{\widetilde {\mathcal X}_1^\circ}\big).
\]
Since \eqref{E:L:VZ03Clm4.2b-proof} is an isomorphism over $U'=\tau^{-1}(\widetilde {\mathcal X}-\tilde {\mathbf \Delta})$, and \eqref{E:L:VZ03Clm4.2b-proof} and \eqref{E:L:PTWClmA.3-proof'} agree over $\tau^{-1}(\widetilde {\mathcal X}-\mathcal S)$, we have that

$$\tilde s|_{\widetilde {\mathcal X}_1^\circ-\mu_1^{-1}\mathcal S}=\eta(\tilde s')|_{\widetilde {\mathcal X}_1^\circ-\mu_1^{-1}\mathcal S}=\mu_1^*s|_{\widetilde {\mathcal X}_1^\circ-\mu_1^{-1}\mathcal S}$$
\end{proof}

\section{Proof of the Main Results}\label{S:ProofMain}

\begin{proof}[Proof of \Cref{T:main-goodMM}]
We start with a family $f:\mathcal Y\to \mathcal X$ as in  \Cref{T:main-goodMM}. 
Let $\widetilde {\mathcal X}\to \mathcal X$ be a log resolution of  $(\mathcal X, \mathbf \Delta)$, and let $\widetilde {\mathcal Y}$ be a strong log-resolution of singularities of the main component of $\widetilde {\mathcal X}\times _{\mathcal X}\mathcal Y$, so that we have a schematic projective morphism $\tilde f:\widetilde {\mathcal Y}\to \widetilde {\mathcal X}$ of smooth proper integral DM stacks over $\mathbb C$ with projective coarse moduli space.  Let $\tilde {\mathbf \Delta}\subseteq \widetilde {\mathcal X}$ be the reduced pre-image of $\mathbf \Delta$.  
Let $\mathcal U=\mathcal X-\mathbf \Delta$ be the complement of the divisor $\mathbf \Delta$; since the morphism $\widetilde {\mathcal X}\to \mathcal X$ is an isomorphism on $\mathcal U$, and $f|_{\mathcal U}:\mathcal Y|_{\mathcal U}\to \mathcal U$ is smooth  (we assumed $\mathbf \Delta_f\subseteq \mathbf \Delta$), 
 it follows that $\mathbf \Delta_{\tilde f}\subseteq \widetilde {\mathcal X}$ is contained in $\tilde {\mathbf \Delta}$.  It follows from \cite[Lem.~1.9]{CMZlog_general_type_KSBA}  that in order to prove that $K_{\mathcal X}+\mathbf \Delta$ is big,  it suffices to prove that $K_{\widetilde {\mathcal X}}+\tilde {\mathbf \Delta}$ is big.  In other words, it suffices to prove  \Cref{T:main-goodMM} under the assumption that $\mathbf \Delta$ is an snc divisor.  Moreover, going forward we will continue to use \cite[Lem.~1.9]{CMZlog_general_type_KSBA}  without mention to replace $\mathcal X$ with a blow-up and $\mathbf \Delta$ with the strict transform of its pre-image.

Now, due to our assumptions on the geometric generic fibers of $f:\mathcal Y\to \mathcal X$,  together with our assumption on the maximal variation of $f$,  the family $f:\mathcal Y\to \mathcal X$ satisfies the hypotheses of \Cref{P:PTW18-PA.1}.   The assertions of \Cref{P:PTW18-PA.1} imply that, after a further blow-up of $\mathcal X$, we may assume that  $f:\mathcal Y\to \mathcal X$ and $\mathbf \Delta$ satisfy the hypotheses of \cite[\S 4.1.1--2]{CMZlog_general_type_KSBA}. \Cref{P:PTW18-PA.4} implies that $f:\mathcal Y\to \mathcal X$ satisfies \cite[\S 4.1.3]{CMZlog_general_type_KSBA}, as well; see especially 
 \cite[Rem.~4.1]{CMZlog_general_type_KSBA} regarding \cite[\S 4.1.3(2)]{CMZlog_general_type_KSBA}, and also observe that, in the notation of   \cite[\S 4.1.3]{CMZlog_general_type_KSBA}, one may take the divisor $\mathcal S'$ to be any divisor containing both $\mathcal S$ as well as the branch divisor $\mathbf \Delta_\tau$, with $\mathbf \Delta_\tau$ as described in \Cref{P:PTW18-PA.4}.   Note that in addition to the family $f:\mathcal Y\to \mathcal X$ and the snc divisor $\mathbf \Delta$, we now have  the additional data of a  line bundle $\mathcal A$ on $\mathcal X$ such that $\mathcal A(-\mathbf \Delta)$ is big.

Now that we have $f:\mathcal Y\to \mathcal X$ satisfying the conditions in 
 \cite[\S 4.1.1--3]{CMZlog_general_type_KSBA}, the result 
 \cite[Thm.~4.13]{CMZlog_general_type_KSBA} 
implies that, after again replacing $\mathcal X$ with a further blow-up,  there is a closed substack $ {\mathbf T}\subseteq  {\mathcal X}$ of codimension at least $2$, so that   setting ${\mathcal X}^\circ:= {\mathcal X}- {\mathbf T}$,        there  is a  Hodge module ${\mathsf M}$ on $ {\mathcal X}$ with strict support $ {\mathcal X}$ and a graded ${\mathscr{A}}_{{\XX}^\circ}$-module $ {\m{G}}^\circ_{\bullet}$ that is coherent over $\mathscr{A}_{ {\XX}^\circ}^\circ$, so that ${\mathsf M}^\circ := {\mathsf M}|_{ {\mathcal X}^\circ}$ and  ${\m{G}}^\circ_{\bullet}$ satisfy  \cite[Thm.~4.3(b)--(e)]{CMZlog_general_type_KSBA} 
   on $ {\mathcal X}^\circ$ with $\mathcal A$ replaced with ${\mathcal A}^\circ:={\mathcal A}|_{ {\mathcal X}^\circ}$.   Moreover,  $\mathcal A(-\mathbf \Delta)$ is big, and  
 letting ${\mathcal S}$ be the singular support of $ {\mathsf M}$,  the divisor $\mathbf E:={\mathbf \Delta}\cup {\mathcal S}$ is snc.

We may therefore employ  
 \cite[Thm.~5.1]{CMZlog_general_type_KSBA}, 
and obtain graded sheaves $\mathcal F_\bullet\subseteq  \mathcal E_\bullet$ on $ {\mathcal X}$, with $\mathcal F_\bullet$ a graded sub-module of a  graded logarithmic Higgs bundle $\mathcal E_\bullet$ that extends a variation of Hodge structure on ${\mathcal X}-{\mathbf E}$, with the property that $\theta_\bullet (\mathcal F_\bullet)\subseteq \Omega^1_{{\mathcal X}}(\log {\mathbf \Delta}) \otimes \mathcal F_{\bullet +1}$, and $\mathcal F_0$ is a big line bundle.  This data allows us to use \cite[Thm.~B]{CMZpositivity} to obtain a coherent sheaf $\mathcal H$ on ${\mathcal X}$ with big determinant, and an inclusion 
$$
\xymatrix{
\mathcal H \ar@{^(->}[r]& \left(\Omega^1_{ {\mathcal X}}(\log  {\mathbf \Delta})\right)^{\otimes s}
}
$$
for some positive integer $s$.
Finally, from \cite[Cor.~B]{CMZfoliations}, such an inclusion implies $K_{ {\mathcal X}}+ {\mathbf \Delta}$ is big, completing the proof. 
\end{proof}

\begin{proof}[Proof of \Cref{main-coarseMS-cor}]
For any divisor $H$, the divisor with superscript $H^{st}$ will mean the strict transform of the divisor $H$ under the birational morphism in consideration.  
As in the start of the proof of \Cref{T:main-goodMM}, let  $\mu:\widetilde {\mathcal X}\to \mathcal X$ be a log resolution of  $(\mathcal X, \mathbf \Delta)$, and let $\widetilde {\mathcal Y}$ be a strong log-resolution of singularities of the main component of $\widetilde {\mathcal X}\times _{\mathcal X}\mathcal Y$. 
Set  $\mathbf\Delta':=\mathbf\Delta^{st}+\operatorname{Ex}(\mu)$. 
Applying \Cref{T:main-goodMM}, we obtain that $K_{\X'}+\mathbf\Delta'$ is big. At the level of the coarse moduli space, this implies that $K_{X'}+R'+\Delta'$ is big; from this we obtain  that $(X',R' +\Delta')$ is log general type  since $(X',R' +\Delta')$ is a log canonical pair. Since our resolutions are schematic, we do not change the order of the stabilizers along the irreducible components of the boundary divisors. Consequently, we have   $R'+\Delta'\leq R^{st}+\Delta^{st}+ \operatorname{Ex}(\mu)$, where here, overloading the notation, $\mu$ denotes the morphism of coarse moduli spaces.  This implies that $(X,R+\Delta)$ is of log general type.
\end{proof}

\section{Examples}\label{S:examples}

In this section we apply \Cref{T:main-goodMM} and \Cref{main-coarseMS-cor} to some standard examples, and compare the outcome to results in the literature. 

\subsection{Moduli of curves}\label{S:MgApplication}
Consider $\overline{\mathcal M}_g$, the moduli stack of Deligne--Mumford stable curves of genus $g\ge 2$.  Applying   \Cref{T:main-goodMM} (for $g\ge 3$,  where the generic stabilizer is trivial) or \cite[Thm.~A]{CMZlog_general_type_KSBA},  we obtain the well-known result that  $K_{\overline{\mathcal M}_g}+\boldsymbol \delta$ is big, where $\boldsymbol \delta=\boldsymbol \delta_0+\boldsymbol \delta_1+\cdots+\boldsymbol \delta_{\lfloor g/2\rfloor}$ is the total boundary divisor; recall that \cite[Thm.~(1.3)]{cornalba_harris} implies the stronger statement that $K_{\overline{\mathcal M}_g}+\boldsymbol \delta$ is in fact the pull back of an ample $\mathbb Q$-line bundle on the coarse moduli space $\overline M_g$.  

In fact, more recently, it has been shown that  $K_{\overline{\mathcal{M}}_g} + \boldsymbol \delta_0$ is big for $g\geq 2$  \cite{CLV16}.
The argument, using the terminology of \cite{CLV16},  is as follows: The divisor $\mathbf D:=K_{\overline{\mathcal{M}}_g} + \boldsymbol \delta_0 = 13\boldsymbol \lambda - \boldsymbol \delta_0 - 2\sum_{i=1}^{\lfloor g/2 \rfloor} \boldsymbol \delta_i $ satisfies the conditions of being a strict M-divisor \cite[p.533]{CLV16}  as soon as $g\geq 2$.   The result \cite[Thm.~1(b)]{CLV16} implies, for the augmented base locus $\mathbf B_+(\mathbf D)$, that  $\mathbf B_+(\mathbf D) \neq \overline{\mathcal{M}}_g$, and therefore that $\mathbf D$ is big \cite[p.534]{CLV16}.

We can also use \Cref{T:main-goodMM} to show that $K_{\overline{\mathcal M}_g}+\boldsymbol \delta_0$ is big:

\begin{cor}[{\cite{CLV16}}]\label{C:K+delta0}
For $g\ge 2$ one has $K_{\overline{\mathcal M}_g}+\boldsymbol \delta_0$ is big.
\end{cor}

\begin{proof}
For all $g\ge 2$, we can consider the Torelli map $\overline{\mathcal M}_g\to \bar {\mathcal A}^{Vor}_g$, to the second Voronoi compactification.  For $g\ge 3$, the generic point of $\overline{\mathcal M}_g$ has trivial stabilizers, and so, if we pull back Alexeev's universal family of stable semi-abelic pairs \cite{alexeev},   we may apply  \Cref{T:main-goodMM}, using that the geometric generic fiber is $K$-trivial, and taking $\mathbf \Delta = \boldsymbol \delta_0$.  It follows that $K_{\overline{\mathcal M}_g}+\boldsymbol \delta_0$ is big.

For $g=2$, one can check directly.  Indeed, since $K_{\overline{\mathcal M}_g}=13\boldsymbol \lambda -2\boldsymbol \delta$, and on $\overline{\mathcal M}_2$ one has $10\boldsymbol \lambda =\boldsymbol \delta_0+2\boldsymbol \delta_1$ (e.g., \cite[Exe.~3.143]{harris_morrison}), one has that $K_{\overline{\mathcal M}_2}+\boldsymbol \delta_0=13\boldsymbol \lambda -\boldsymbol \delta_0-2\boldsymbol \delta_1=3\boldsymbol \lambda$, which is big.  Alternatively, using the isomorphism $\overline {\mathcal M}_2\cong \bar{\mathcal A}_2^{Vor}$, one has $K_{\overline{\mathcal M}_2}+\boldsymbol \delta_0=K_{\bar{\mathcal A}_2^{Vor}}+\mathbf \Delta^{tor}=3\boldsymbol\lambda$, where $\mathbf \Delta^{tor}$ is the toroidal boundary divisor.
\end{proof}

\begin{rem}
For $g=3$, it is also elementary to check directly that $K_{\overline{\mathcal M}_3}+\boldsymbol \delta_0$ is big.  Here we use that the divisor $\overline {\mathbf h}$, the closure of the hyperelliptic locus,  has class $\overline {\mathbf h} = 9\boldsymbol \lambda -\boldsymbol \delta_0 -3\boldsymbol \delta_1$ (e.g., \cite[p.188]{harris_morrison}).  Then we have $K_{\overline{\mathcal M}_3}+\boldsymbol \delta_0=13\boldsymbol \lambda -\boldsymbol \delta_0-2\boldsymbol \delta_1=\overline{\mathbf  h} +4\boldsymbol \lambda +\boldsymbol \delta_1$, which is big plus effective, and therefore big.  
\end{rem}

For $g\ge 4$, the ramification divisor for $\pi:\overline {\mathcal M}_g\to \overline M_g$ is $\boldsymbol \delta_1$, with integral branch divisor $\Delta_1$; here we have $\pi^*\frac{1}{2}\Delta_1=\boldsymbol \delta_1$.   We let $\Delta_0$ be the integral divisor on $\overline M_g$ such that $\pi^*\Delta_0=\boldsymbol \delta_0$.  \Cref{C:K+delta0} then gives us that $K_{\overline M_g}+\frac{1}{2}\Delta_1+\Delta_0$ is big.
For $g=3$,  the branch divisor also includes the hyperelliptic divisor $H$.   \Cref{C:K+delta0} then gives that $K_{\overline M_3}+\frac{1}{2}H+\frac{1}{2}\Delta_1+\Delta_0$ is big. 
For $g=2$, the branch divisor is $\Delta_1\cup B$, where $B$ is the closure of the locus of bielliptic curves.   \Cref{C:K+delta0} then gives that $K_{\overline M_2}+\frac{1}{2}B+\frac{1}{2}\Delta_1+\Delta_0$ is big.

\subsection{Moduli of marked admissible covers}
We briefly note here that for any smooth DM moduli stack $\overline {\mathcal R}$ with projective coarse moduli space  that admits a generically finite morphism $\overline{\mathcal R}\to\overline {\mathcal M}_{g,n}$ for $g,n\ge 0$, $(g,n)\notin \{(0,0), (0,1),(0,2), (1,0)\}$ (or to any of the KSBA compactifications of the moduli space of marked curves, or, more generally, to any KSBA moduli space), one can apply either    \Cref{T:main-goodMM}  or \cite[Thm.~A]{CMZlog_general_type_KSBA}  to $\overline{\mathcal R}$.  In particular, one can apply the theorems to moduli spaces of admissible covers of curves; we refer the reader to Abramovich--Corti--Vistoli \cite{ACV03} for a discussion of various moduli stacks of curves with covers.   

As a concrete example, consider the moduli space of admissible double covers $\overline{\mathcal R}_g$, compactifying the moduli space of \'etale double covers of smooth projective curves of genus $g\ge 2$.   An application of  \cite[Thm.~A]{CMZlog_general_type_KSBA} implies that $K_{\overline{\mathcal R}_g}+\boldsymbol \delta_{\overline{\mathcal R}_g}$ is big, where $\boldsymbol \delta_{\overline{\mathcal R}_g}$ is the full boundary divisor; see \cite{FL10Prym}  for a description of the boundary divisor.  For $g\ge 4$, we also refer the reader to  \cite[Prop.~6.6]{FL10Prym}  for a description of the ramification divisor of the stack, which lies over $\boldsymbol \delta_1$ of $\overline{\mathcal M}_g$.  
As the ramification divisor is contained in the discriminant for $g\ge 4$, we obtain that $K_{\overline R_g}+\Delta_{\overline{\mathcal R}_g}$ is big, where $\Delta_{\overline{\mathcal R}_g}$ is the full (reduced) boundary divisor on $\overline R_g$. 
Note also  that as the finite morphism  $f:\overline{\mathcal R}_g\to \overline{\mathcal M}_g$ is ramified over the boundary (the ramification divisor lies over $\boldsymbol \delta_0$ in $\overline {\mathcal M}_g$; see \cite[p.763]{FL10Prym}),  we have that $K_{\overline{\mathcal R}_g}+\boldsymbol \delta_{\overline{\mathcal R}_g}=f^*(K_{\overline{\mathcal M}_g} +\boldsymbol \delta)$, so that in fact, $K_{\overline{\mathcal R}_g}+\boldsymbol \delta_{\overline{\mathcal R}_g}$ is the pull back of an ample $\mathbb Q$-line bundle on the coarse moduli space $\overline R_g$.

\subsection{Moduli of abelian varieties}
Here we consider $\bar {\mathcal A}_g^{Vor}$, the second Voronoi compactification of the moduli stack of principally polarized abelian varieties, and view it as (the normalization of)  Alexeev's moduli space of stable semi-abelic pairs \cite{alexeev}.  Let $\bar {\mathcal A}_g^{tor}\to \bar {\mathcal A}_g^{Vor}$ be a toroidal compactification (stack) dominating the second Voronoi compactification, which we may, if we want, assume is smooth (see, e.g., \cite[IV Cor.~2.4]{AMRT10});  
note that for $g \le 4$, one has that $ \bar {\mathcal A}_g^{Vor}$ is smooth \cite[Thm.~1.4(i)]{DSHS15}.
We then consider the pull back to $\bar {\mathcal A}_g^{tor}$ of the universal family over $\bar {\mathcal A}_g^{Vor}$.  

If we assume $\bar {\mathcal A}_g^{tor}$ is smooth, and view it as the base of a family of generically $K$-trivial varieties of maximal variation, we cannot  apply 
\Cref{T:main-goodMM}, as generically there is a $\mathbb Z_2$-automorphism group.  Nevertheless, it is natural to ask if the conclusion of the theorem still holds, namely,  whether $K_{\bar{\mathcal A}_g^{tor}}+\mathbf \Delta^{tor}$ is big, where $\mathbf \Delta^{tor}$ is the toroidal boundary divisor, which parameterizes singular stable semi-abelic pairs (i.e., the pairs $(P,\Theta)$, where $P$ is singular). In fact, whether or not $\bar {\mathcal A}_g^{tor}$ is smooth, using Hirzebruch--Mumford proportionality \cite[Prop.~3.4]{MumfordHMP77}, we have the stronger statement that $K_{\bar {\mathcal A}_g^{tor}}+\mathbf \Delta^{tor}=(g+1)\boldsymbol \lambda$, where $\boldsymbol \lambda$ is the pull back of the Hodge  bundle under the composition $\bar{\mathcal A}_g^{tor}\to \bar A_g^{tor}\to A_g^*$, and  $A_g^*$ is the Satake compactification.  Consequently, we see that  $K_{\bar {\mathcal A}_g^{tor}}+\mathbf \Delta^{tor}$ is nef and big (even the pull back of a semi-ample $\mathbb Q$-line bundle on the coarse moduli space).  

On the other hand, if we assume again that $\bar{\mathcal A}_g^{tor}$ is smooth, and view it as the base of a family of KSBA stable pairs of maximal variation, then if we apply \cite[Thm.~A]{CMZlog_general_type_KSBA},  we obtain that $K_{\bar {\mathcal A}_g^{tor}}+\mathbf \Delta$ is big, where  $\mathbf \Delta = \mathbf \Delta^{tor}+(\mathbf N_0)_{\operatorname{red}}$, and  $\mathbf N_0$ is the Andreotti--Mayer divisor compactifying the locus of ppavs with a singular theta divisor.  
In fact, whether or not $\bar{\mathcal A}_g^{tor}$ is smooth, from Hirzebruch--Mumford proportionality we have that $$K_{\bar {\mathcal A}_g^{tor}}+\mathbf \Delta=K _{\bar {\mathcal A}_g^{tor}}+\mathbf \Delta^{tor}+(\mathbf N_0)_{\operatorname{red}}=(g+1)\boldsymbol \lambda + (\mathbf N_0)_{\operatorname{red}}$$ is big plus effective, and therefore is big.  
 However, in contrast to the previous paragraph, as we will see below, $K_{\bar {\mathcal A}_g^{tor}}+\mathbf \Delta$ is not nef  for $g\ge 3$; in particular, $K_{\bar {\mathcal A}_g^{Vor}}+\mathbf \Delta$ is not nef  for $g\ge 3$. 
 
 \begin{rem}
 For $g=3,4$,  the stack $\bar {\mathcal A}_g^{Vor}$  provides an example of a smooth proper integral DM stack $\mathcal X$ over $\mathbb C$ with projective coarse moduli space, admitting a finite morphism to a KSBA moduli stack, having relative snc discriminant $\mathbf \Delta\subseteq \mathcal X$ a divisor (see \cite[Def.~1.1]{CMZlog_general_type_KSBA}), with the property  that $K_{\mathcal X}+\mathbf \Delta$ is big,  but not nef. 
\end{rem}

To show that $K_{\bar {\mathcal A}_g^{Vor}}+\mathbf \Delta$ is not nef  for $g\ge 3$,  recall from \cite[p.368]{MumKodAg} that $\mathbf N_0=\boldsymbol \theta_{null}+2\mathbf N_0'$, where $\boldsymbol \theta_{null}$  is the divisor compactifying the locus of ppavs with a vanishing theta null (at the generic point of  $\boldsymbol \theta_{null}$, the theta divisor has a unique ordinary double point), and  $\mathbf N_0'$  is the divisor compactifying the locus of ppavs whose theta divisor has exactly two ordinary double points \cite{debarre92}.   Letting $\mathcal A_g'\subseteq \bar{\mathcal A}_g^{tor}$ be Mumford's partial compactification corresponding to rank $1$ torus degenerations, Mumford showed \cite[Cor.~1.6]{MumKodAg} that $\operatorname{Pic}(\mathcal A_g')\otimes_{\mathbb Z}\mathbb Q= \mathbb Q\langle \boldsymbol \lambda, \mathbf D\rangle$, where $\mathbf D\subseteq \mathcal A_g'$ is the boundary divisor, and that on $\mathcal A_g'$  one has  \cite[Thm.~2.10]{MumKodAg}:
\begin{align*}
[\boldsymbol \theta_{\mathrm{null}}]
&=
2^{g-2}(2^g+1)\boldsymbol \lambda
\;-\;
2^{2g-5}\mathbf D,\\
[\mathbf N_0']
&=
\left[
\frac{(g+1)!}{4}
+
\frac{g!}{2}
-
2^{g-3}(2^g+1)
\right]\boldsymbol \lambda
\;-\;
\left[
\frac{(g+1)!}{24}
-
2^{2g-6}
\right]\mathbf D.
\end{align*}  
Recall that for $g=1$ one has that $\boldsymbol \theta_{null}$ is empty, and for   $g\le 3$ one has that $\mathbf N_0'$ is empty; the formulas above also hold for $g$ in this range (for $g=1$, use that $\mathbf D=12\boldsymbol \lambda$).  Note also that for $g=3$, the Torelli map $\overline{\mathcal M}_3\to \bar{\mathcal A}_3^{Vor}$ is ramified with order $2$ along the hyperelliptic divisor $\overline {\mathbf h}$, with branch divisor $\boldsymbol \theta_{null}$, and that this accounts for the factor-of-two difference in the classes; i.e., $\overline{\mathbf h}=9\boldsymbol \lambda -\boldsymbol\delta_0-3\boldsymbol\delta_1$, whereas $\boldsymbol\theta_{null}=18\boldsymbol\lambda - 2\mathbf \Delta^{tor}$.  

From the formulas above, one has 
$$
(K_{\bar {\mathcal A}_g^{tor}}+\mathbf \Delta)|_{\mathcal A_g'}=(K_{\bar {\mathcal A}_g^{tor}}+\mathbf \Delta^{tor}+\boldsymbol \theta_{null}+\mathbf N_0')|_{\mathcal A_g'}=((g+1)\boldsymbol \lambda +\boldsymbol \theta_{null}+\mathbf N_0')|_{\mathcal A_g'}
$$
\[
=
\left[(g+1)+
\frac{(g+1)!}{4}
+\frac{g!}{2}
+2^{g-3}(2^g+1)
\right]\boldsymbol \lambda
-
\left[
\frac{(g+1)!}{24}
+
2^{2g-6}
\right]\mathbf D.
\]
Now consider the test curve $\mathbf T=\mathbb P^1\subseteq \mathcal A_g'$ given by taking the product of a fixed generic ppav of dimension $g-1$ with the compactified Jacobian of a generic pencil of plane cubics.  One can see that $\mathbf D|_{\mathbf T}=\mathcal O_{\mathbb P^1}(12)$ and $\boldsymbol \lambda|_{\mathbf T}= \mathcal O_{\mathbb P^1}(1)$.  Consequently, for any line bundle $\mathcal L$ on $\bar{\mathcal A}_g^{tor}$ with restriction  $\mathcal L|_{\mathcal A_g'}= a\boldsymbol\lambda -b\mathbf D$, then for $\mathcal L$ to be nef, one must have $a\ge 12 b$.  In fact, since $\mathbf D$ is contracted under the morphism to the Satake compactification, one can also deduce that $b\ge 0$; we refer the reader to \cite{HSnef04, HulekNefDiv2000} for more on the nef cone of $\bar{\mathcal A}_g^{tor}$.   For $a\boldsymbol \lambda -b\mathbf D$, we define the slope as $s(a\boldsymbol \lambda -b\mathbf D)=a/b$, and so, translating the nef condition above into a statement about slopes, we have that if $\mathcal L$ is nef, then  $s(\mathcal L|_{\mathcal A_g'})\ge 12$. 

 Now plugging $g=3,4$ into the formula above for $(K_{\bar {\mathcal A}_g^{tor}}+\mathbf \Delta)|_{\mathcal A_g'}$, we have that 
 \begin{align*}
( K_{\bar {\mathcal A}_3^{tor}}+\mathbf \Delta)|_{\mathcal A_3'}&= 22\lambda -2\mathbf D, \\
( K_{\bar {\mathcal A}_4^{tor}}+\mathbf \Delta)|_{\mathcal A_4'}&= 81\lambda -9\mathbf D,
\end{align*}
 so that $K_{\bar {\mathcal A}_g^{tor}}+\mathbf \Delta$ is not nef for $g=3,4$.  With a little work, one can show that  $s(K_{\bar {\mathcal A}_{g+1}^{tor}}+\mathbf \Delta)<s(K_{\bar {\mathcal A}_g^{tor}}+\mathbf \Delta)$ for all $g\ge 4$, and so it follows that $K_{\bar {\mathcal A}_g^{tor}}+\mathbf \Delta$ is not nef for $g\ge 3$.

\begin{rem}[Low dimension case] For $g\le 2$, we have in contrast that $K_{\bar {\mathcal A}_g^{Vor}}+\mathbf \Delta$ is the pull back of an ample $\mathbb Q$-line bundle on the coarse moduli space $\bar A_g^{Vor}$, and in particular is nef.   
 For $g=1$, we have $\bar {\mathcal A}_1^{Vor}\cong \overline {\mathcal M}_{1,1}$, and  $K_{\bar {\mathcal A}_1^{Vor}}+\mathbf \Delta= K_{\overline {\mathcal M}_{1,1}}+\boldsymbol \delta$ since $\mathbf N_0$ is empty and $\mathbf \Delta^{tor}=\boldsymbol \delta$; one can check directly (see, e.g., the introduction to \cite{CMZlog_general_type_KSBA}, or the discussion below) that $K_{\overline {\mathcal M}_{1,1}}+\boldsymbol \delta$ is the pull back of an ample $\mathbb Q$-line bundle on the coarse moduli space.
 For $g=2$, we have  $\bar {\mathcal A}_2^{Vor}\cong \overline {\mathcal M}_{2}$, and  $K_{\bar {\mathcal A}_2^{Vor}}+\mathbf \Delta= K_{\overline {\mathcal M}_{2}}+\boldsymbol \delta$ since $\mathbf N_0=\boldsymbol\theta_{null}=\boldsymbol \delta_1$ and $\mathbf \Delta^{tor}=\boldsymbol \delta_0$;  from \cite[Thm.~(1.3)]{cornalba_harris} we have  that $K_{\overline{\mathcal M}_2}+\boldsymbol\delta$ is the pull back of an ample $\mathbb Q$-line bundle on the coarse moduli space.  Alternatively, one can consider that $\operatorname{Pic}(\bar {\mathcal A}_2^{Vor})\otimes_{\mathbb Z}\mathbb Q= \operatorname{Pic}(\mathcal A_2')\otimes_{\mathbb Z}\mathbb Q$, and compute from the formulas above that 
$K_{\bar {\mathcal A}_2^{Vor}}+\mathbf \Delta= K_{\bar {\mathcal A}_2^{Vor}}+\mathbf \Delta^{tor}+\boldsymbol \theta_{null}= 3 \boldsymbol \lambda + 5\boldsymbol \lambda -\frac{1}{2}\mathbf D= 8\boldsymbol \lambda -\frac{1}{2}\mathbf D$.  Then, using \cite[Thm.~0.2]{HulekNefDiv2000}, one has that $K_{\bar {\mathcal A}_2^{Vor}}+\mathbf \Delta$ is in the interior of the nef cone, as its slope is greater than $12$, and so it is the pull back of an ample $\mathbb Q$-line bundle on the coarse moduli space.
\end{rem}

Regarding the corresponding results on the coarse moduli spaces, recall from  
\cite[p.429]{TaiKodaira82} that the ramification divisor for the coarse moduli morphism  $\bar{\mathcal A}_g^{tor}\to \bar A^{tor}_g$ is empty for $g\ge 3$.  The results above then say that for $g\ge 3$, 
$$K_{\bar A_g^{tor}}+\Delta=K_{\bar A_g^{tor}}+\Delta^{tor}+\theta_{null}+(N_0')_{\operatorname{red}}$$
 is big, but not nef.   

For $g=2$, under the isomorphism  $\bar {\mathcal A}_2^{Vor}\cong \overline {\mathcal M}_{2}$, we have 
 $\mathbf \Delta^{tor}=\boldsymbol \delta_0$ and $\mathbf N_0=\boldsymbol\theta_{null}=\boldsymbol \delta_1$, so that $\mathbf \Delta=\boldsymbol\delta$.  The ramification divisor for $\bar{\mathcal A}_2^{Vor}\to \bar A^{Vor}_2$ is $\boldsymbol\theta_{null}\cup\boldsymbol \beta$, where $\boldsymbol\beta$ is the closure of the locus of Jacobians of bi-elliptic curves; at the generic point of both divisors there is an extra $\mathbb Z_2$-automorphism group.  At the level of the Siegel upper half-space, in coordinates 
$
\left(
\begin{array}{cc}
\tau_1& \tau_2\\
\tau_2& \tau_3
\end{array}
\right)
$, these correspond to the divisors $\tau_2=0$, and $\tau_1-\tau_3=0$, respectively, giving rise to involutions in $\operatorname{Sp}(4,\mathbb Z)$ (see \cite{HWbielliptic89}).  
In other words, we have $R=\frac{1}{2}\theta_{null}+\frac{1}{2}\beta$, where 
 $\beta$ on $\bar A_2^{Vor}$ is the reduced divisor given as the closure of the locus of Jacobians of bi-elliptic curves, and $\Delta = \Delta^{tor}+\frac{1}{2}\theta_{null}$.  
The results above then say that 
$$K_{\bar A_2^{Vor}}+R+\Delta=K_{\bar A_2^{Vor}}+\frac{1}{2}\beta+\Delta^{tor}+\theta_{null}$$ is ample.

For $g=1$, we have $\bar {\mathcal A}_1^{Vor}\cong \overline {\mathcal M}_{1,1}$ and  $\mathbf \Delta^{tor}=\boldsymbol \delta$.
Letting $p_j$ denote the point of $\overline M_{1,1}$ parameterizing elliptic curves with $j$-invariant $j$, we have that  the branch divisor is  given by  $R=(1-\frac{1}{3})p_{0}+ (1-\frac{1}{2})p_{1728}$, and the boundary divisor is given by $ \Delta = p_\infty$.  The results above then say that
$
K_{\bar A_1^{Vor}}+R+\Delta=K_{\mathbb P^1} +\frac{2}{3}p_{0}+ \frac{1}{2}p_{1728}+p_\infty$
is ample.

\subsection{Moduli of cubic surfaces}

The KSBA moduli space $\overline{\mathcal M}_{\operatorname{cub}}$ of cubic surfaces \cite{GKS,HKT09}, with coarse moduli space $\overline M
_{\operatorname{cub}}$,   provides another interesting example, as it can be described in several different ways, giving different ways of applying  \Cref{T:main-goodMM}, as well as \cite[Thm.~A]{CMZlog_general_type_KSBA}.

As $\overline{\mathcal M}_{\operatorname{cub}}$ is given as a moduli space of pairs, first we explain how to apply  \cite[Thm.~A]{CMZlog_general_type_KSBA} to the moduli space. 
Recall that the interior $\mathcal M_{\operatorname{cub}}$ of 
 $\overline{\mathcal M}_{\operatorname{cub}}$ parameterizes pairs $(S,(\frac{1}{9}+\epsilon) D)$,  $0<\epsilon\ll 1$, where $S$ is a smooth cubic surface and $D\subseteq S$ is the union of the $27$ lines on the cubic surface.   The boundary of  $\overline{\mathcal M}_{\operatorname{cub}}$ has two irreducible components, a divisor $\mathbf \Delta_{A_1}$ whose generic point parameterizes cubics with a single singularity, which is of type $A_1$, and a divisor $\mathbf \Delta_{3A_2}$, whose generic point parameterizes three hyperplanes in $\mathbb P^3_{\mathbb C}$ each containing $9$ lines with weight $\frac{1}{9}+\epsilon$ (the terminology $\mathbf \Delta_{3A_2}$ comes from the connection with the related GIT moduli space).  
 We refer the reader to \cite{GKS} for a more precise description of the boundary divisors for the singular surfaces.   The divisorial locus in $\overline{\mathcal M}_{\operatorname{cub}}$ parameterizing surfaces with  extra automorphisms is the irreducible Eckardt divisor $\mathbf R$, whose generic point   corresponds to a smooth cubic surface with an Eckardt point.  Recall that an Eckardt point is a point on a cubic surface where $3$ lines intersect, and a general cubic surface with an Eckardt point has  a $\mathbb Z_2$ automorphism group.  
 We denote by  $\pi:\overline{\mathcal M}_{\operatorname{cub}}\to \overline{M}_{\operatorname{cub}}$ the canonical morphism to the coarse moduli space, and we let $\Delta_{A_1}$, $\Delta_{3A_2}$, and $R$ be the $\mathbb Q$-divisors on $\overline{M}_{\operatorname{cub}}$ so that $\pi^*\Delta_{A_1}=\mathbf \Delta_{A_1}$, $\pi^*\Delta_{3A_2}=\mathbf \Delta_{3A_2}$, and $\pi^*R=\mathbf R$.   We denote by $R'$ the integral divisor with the same support as $R$, which gives us that $R=\frac{1}{2}R'$.

As a smooth moduli stack of KSBA stable pairs, we can apply  \cite[Thm.~A]{CMZlog_general_type_KSBA} to $\overline{\mathcal M}_{\operatorname{cub}}$, and the above discussion implies that $K_{ \overline{\mathcal M}_{\operatorname{cub}}}+\mathbf R+\mathbf \Delta_{A_1}+\mathbf \Delta_{3A_2}$ is big; note that here $\mathbf R$ is considered part of the pair discriminant, since the three lines intersecting in the Eckardt point are not normal crossing on the surface.  This implies that
\begin{equation}\label{E:ModCub1}
K_{\overline{M}_{\operatorname{cub}}}+R'+\Delta_{A_1}+\Delta_{3A_2}
\end{equation}
is big; incidentally, this is the same result one would obtain by applying  \cite[Thm.~A]{WW23} to $\overline M_{\operatorname{cub}} - (R'\cup \Delta_{A_1}\cup \Delta_{3A_2})$.

An alternate approach to the moduli space of cubic surfaces can be given via the Allcock--Carlson--Toledo \cite{ACTsurf} construction, taking the triple cover of $\mathbb P^3_{\mathbb C}$ branched along the cubic surface, which yields a cubic threefold with a $\mu_3$-action.  Taking the intermediate Jacobian gives a period map $\mathcal M_{\operatorname{cub}}\to \mathcal A_5$, and using the Clemens--Griffiths Torelli theorem for cubic threefolds, one has that the pull back of the universal family of principally polarized abelian varieties has maximal variation.  After resolving the period map to the second Voronoi compactification,  an application of  \Cref{T:main-goodMM} 
to the resolution (note that even though we have a family of abelian varieties,  the base of the family is a stack that does not have nontrivial stabilizers at the generic point),  implies,
after also invoking \cite[Lem. 1.9]{CMZlog_general_type_KSBA}, 
 that $K_{ \overline{\mathcal M}_{\operatorname{cub}}}+\mathbf \Delta_{A_1}+\mathbf \Delta_{3A_2}$ is big.   This implies that
\begin{equation}\label{E:ModCub2}
K_{\overline{M}_{\operatorname{cub}}}+\frac{1}{2}R'+\Delta_{A_1}+\Delta_{3A_2}
\end{equation}
is big.

  In fact, following the Allcock--Carlson--Toledo construction further by considering the $\mu_3$-action on the intermediate Jacobian, one obtains a ball quotient model of the moduli space of cubic surfaces \cite{ACTsurf}, and a result of \cite{GKS} implies that $\overline{M}_{\operatorname{cub}}$ is isomorphic to the toroidal compactification of this ball quotient.  
  Interestingly, from the perspective of Hodge theory and arithmetic quotients, one has from Hirzebruch--Mumford proportionality that 
  $K_{ \overline{\mathcal M}_{\operatorname{cub}}}+\mathbf \Delta_{3A_2}$ is big.  This implies that 
  \begin{equation}\label{E:ModCub3}
  K_{\overline M_{\operatorname{cub}}}+\frac{1}{2}R'+\frac{5}{6}\Delta_{A_1}+\Delta_{3A_2}
  \end{equation}
   is big and nef (it is a multiple of the pull back of the Hodge bundle from the SBB compactification); the difference in the formula here from that in \eqref{E:ModCub2} can be interpreted as coming from the fact that from the geometric construction of \cite{ACTsurf}, the   Hodge structure associated with cubic surfaces with a unique $A_1$ singularity is pure, and has an order $6$ automorphism not coming from an automorphism of the cubic surface, so that $\Delta_{A_1}$ is part of the ramification locus, rather than the boundary, for the ball quotient.  

Finally, for context, we recall a few results from 
\cite{CMGH25} regarding the nef and effective cones on $\overline M_{\operatorname{cub}}$.  The starting point is that $\operatorname{Pic}_{\mathbb Q}(\overline M_{\operatorname{cub}})=\mathbb Q\langle \Delta_{A_1},\Delta_{3A_2}\rangle$; for convenience, given a $\mathbb Q$-line bundle $\alpha \Delta_{A_1}+\beta\Delta_{3A_2}$, we define the slope as
$$
\mu(\alpha \Delta_{A_1}+\beta\Delta_{3A_2}):=\frac{\beta}{\alpha}.
$$
Then \cite[Thm.~1.1]{CMGH25} implies that the nef cone is determined by the slope condition $2\le \mu \le 6$, and the effective cone is determined by the slope condition $0\le \mu\le \infty$.   
Moreover, $K_{\overline{M}_{\operatorname{cub}}}=\frac{1}{4}\left(-15 \Delta_{A_1}-26\Delta_{3A_2}\right)$ and $R'=\frac{1}{4}\left(25 \Delta_{A_1}+54\Delta_{3A_2}\right)$; see \cite[Prop.~2.1 and (2.4)]{CMGH25}.
Consequently, 
  $$
  K_{\overline M_{\operatorname{cub}}}+a R'+b\Delta_{A_1}+c\Delta_{3A_2}=
   \frac{1}{4}\left(-15 +25 a+4b\right)\Delta_{A_1} + 
   \frac{1}{4}\left(-26+54a+4c\right) 
\Delta_{3A_2}
  $$
so that 
$$
\mu( K_{\overline M_{\operatorname{cub}}}+a R'+b\Delta_{A_1}+c\Delta_{3A_2})=\frac{-26+54a+4c}{-15 +25 a+4b}.
$$
In particular, considering \eqref{E:ModCub3}, we have $\mu(K_{\overline M_{\operatorname{cub}}}+\frac{1}{2}R'+\frac{5}{6}\Delta_{A_1}+\Delta_{3A_2})=6$, which lies on the boundary of the nef cone, reflecting the fact that $K_{\overline M_{\operatorname{cub}}}+\frac{1}{2}R'+\frac{5}{6}\Delta_{A_1}+\Delta_{3A_2}$ is semi-ample, but not ample, being a multiple of the pull back of the Hodge bundle. Meanwhile, considering \eqref{E:ModCub2} and \eqref{E:ModCub1}, respectively, we have $\mu(K_{\overline{M}_{\operatorname{cub}}}+\frac{1}{2}R'+\Delta_{A_1}+\Delta_{3A_2})= 10/3$ and $\mu(K_{\overline{M}_{\operatorname{cub}}}+R'+\Delta_{A_1}+\Delta_{3A_2})=16/7$, so that both are ample.

\appendix

 \bibliographystyle{amsalpha}
 \bibliography{mhm_bib}

\end{document}